\documentclass[11pt,reqno]{amsart}
\usepackage[a4paper,margin=1.1in]{geometry}
\usepackage{amsmath,amssymb,amsthm,mathtools}
\usepackage{booktabs,array}
\usepackage{enumitem}
\usepackage[protrusion=true,expansion=false]{microtype}
\usepackage[hidelinks]{hyperref}
\usepackage[nameinlink,capitalize]{cleveref}
\numberwithin{equation}{section}
\theoremstyle{plain}
\newtheorem{theorem}{Theorem}[section]
\newtheorem{proposition}[theorem]{Proposition}
\newtheorem{lemma}[theorem]{Lemma}
\newtheorem{corollary}[theorem]{Corollary}
\newtheorem*{theoremA}{Theorem A}
\newtheorem*{theoremB}{Theorem B}
\newtheorem*{theoremC}{Theorem C}
\newtheorem*{theoremD}{Theorem D}
\newtheorem*{theoremE}{Theorem E}
\theoremstyle{definition}
\newtheorem{definition}[theorem]{Definition}
\newtheorem{example}[theorem]{Example}

\theoremstyle{remark}
\newtheorem{remark}[theorem]{Remark}

\newcommand{\C}{\mathbb{C}}
\newcommand{\Q}{\mathbb{Q}}
\newcommand{\R}{\mathbb{R}}
\newcommand{\Z}{\mathbb{Z}}
\newcommand{\kk}{\Bbbk}
\newcommand{\g}{\mathfrak{g}}
\newcommand{\kalg}{\mathfrak{k}}
\newcommand{\p}{\mathfrak{p}}
\newcommand{\pp}{\mathfrak{p}^{+}}
\newcommand{\pmns}{\mathfrak{p}^{-}}

\newcommand{\N}{\mathcal{N}}
\newcommand{\Cr}{C_{r}}
\newcommand{\End}{\operatorname{End}}
\newcommand{\Hom}{\operatorname{Hom}}

\newcommand{\rk}{\operatorname{rk}}

\newcommand{\Sym}{\operatorname{Sym}}
\newcommand{\Spec}{\operatorname{Spec}}
\newcommand{\GL}{\operatorname{GL}}
\newcommand{\Gr}{\operatorname{Gr}}

\newcommand{\Ab}{\mathcal{A}}
\newcommand{\sm}{\mathrm{sm}}

\usepackage{parskip}
\begin{document}

\title[Higher commuting varieties of symmetric pairs]
{Dimensions and maximal-dimensional components of higher commuting
varieties of symmetric pairs}

\author{Santosha Pattanayak}
\address{Department of Mathematics, Indian Institute of Technology Kanpur,
Kanpur, Uttar Pradesh 208016, India}
\email{santosha@iitk.ac.in}

\date{}


\begin{abstract}
Let $\mathfrak p$ be the $(-1)$-eigenspace of a complex semisimple
symmetric pair. For every $r\ge2$, we prove
$\dim C_r(\mathcal N(\mathfrak p))\le
\lfloor r\dim\mathfrak p/2\rfloor$ and show that every irreducible
component is isotropic for a natural family of alternating forms.
For a nonzero Hermitian pair with $\mathfrak g$ simple and
$\mathfrak p=\mathfrak p^-\oplus\mathfrak p^+$, put
$d=\dim\mathfrak p^+$ and
$\ell=\operatorname{rk}(\mathfrak g,\mathfrak k)$. Then
\[
 \dim C_r(\mathcal N(\mathfrak p))=rd,\qquad
 \dim C_r(\mathfrak p)=\max\{rd,\,2d+(r-1)\ell\}.
\]
In this case we also determine the maximal-dimensional components of
$C_r(\mathfrak p)$; for $r\ge3$, the only maximal-dimensional nilpotent
components are $(\mathfrak p^+)^r$ and $(\mathfrak p^-)^r$.
For arbitrary symmetric pairs,
$\dim C_r(\mathcal N(\mathfrak p))=ra+m$ for all large $r$, where $a$ is
the largest dimension of an abelian subspace of $\mathfrak p$ consisting
of nilpotent elements and $m$ is the dimension of the variety of such
subspaces. 
Applications give dimension and component formulas for equivariant
punctual Quot schemes under explicit framing conditions, and bounds
for punctual Gorenstein loci refined by socle parity.
\end{abstract}

\maketitle 

\section{Introduction}\label{sec:intro}

Commuting varieties connect Lie theory with the geometry of
finite-dimensional representations and punctual moduli spaces.
For a finite-dimensional complex Lie algebra $\mathfrak h$ and an
integer $r\ge1$, its $r$-fold commuting variety is
\[
 C_r(\mathfrak h)
 =\{(x_1,\ldots,x_r)\in\mathfrak h^r:
       [x_i,x_j]=0\text{ for all }i,j\}.
\]
If $\mathfrak h$ is semisimple, let $\N(\mathfrak h)$ denote its
nilpotent cone and set
$C_r(\N(\mathfrak h))=C_r(\mathfrak h)\cap\N(\mathfrak h)^r$.
For $\mathfrak h=\mathfrak{gl}_n$, we use the same notation for
the cone of nilpotent matrices. Unless stated otherwise, the
varieties considered below carry their reduced structures.

For two entries, the geometry is governed by classical
irreducibility theorems. Motzkin--Taussky and Gerstenhaber proved
the irreducibility of the variety of commuting pairs of matrices
\cite{MotzkinTaussky,Gerstenhaber}, and Richardson established the
corresponding result for reductive Lie algebras
\cite[Theorem~A]{Richardson}.
For nilpotent matrices, Baranovsky proved that
$C_2(\N(\mathfrak{gl}_n))$ is irreducible of dimension $n^2-1$;
see also Basili \cite{Baranovsky,Basili}.
For a complex semisimple Lie algebra $\mathfrak h$, Premet showed
that $C_2(\N(\mathfrak h))$ is equidimensional of dimension
$\dim\mathfrak h$ and described its irreducible components in
terms of nilpotent orbits \cite{Premet}.

For longer tuples, irreducibility fails in general and the
dimension becomes more difficult to determine. In particular, for
$r\ge4$, the variety $C_r(\mathfrak{gl}_n)$ is irreducible if and
only if $n\le3$ \cite{Gerstenhaber,Guralnick}.
Levy--Ngo--\v{S}ivi\'c obtained the exact formula
\[
 \dim C_r(\N(\mathfrak{gl}_n))
 =(r+1)\left\lfloor\frac{n^2}{4}\right\rfloor
 \qquad(n\ge4,\ r\ge7)
\]
over algebraically closed fields of characteristic different
from $2$ and $3$ \cite[Theorem~4.8]{LevyNgoSivic}.
Under the same restrictions,
\[
 \dim C_r(\mathfrak{gl}_n)
 =(r+1)\left\lfloor\frac{n^2}{4}\right\rfloor+r,
\]
except that $\dim C_7(\mathfrak{gl}_4)=40$
\cite[Corollary~4.9]{LevyNgoSivic}.
Their work also formulates a conjecture predicting eventual
linear growth in terms of varieties of elementary subalgebras.

In this paper we study higher commuting varieties associated with
complex semisimple symmetric pairs. Let $\theta$ be an involution
of a complex semisimple Lie algebra $\g$, with eigenspace
decomposition
\[
 \g=\kalg\oplus\p,\qquad
 \kalg=\g^\theta,\qquad
 \p=\{x\in\g:\theta(x)=-x\}.
\]
Let $G$ be the adjoint group of $\g$, let
$K=(G^\theta)^\circ$, and put $\N(\p)=\N(\g)\cap\p$.
Although $\p$ need not be a Lie subalgebra, the commuting equations
define closed subvarieties
\[
 C_r(\p)=C_r(\g)\cap\p^r,\qquad
 C_r(\N(\p))=C_r(\g)\cap\N(\p)^r.
\]
A Cartan subspace $\mathfrak a\subseteq\p$ is a maximal abelian
subspace consisting of semisimple elements. Its dimension
$\ell=\rk(\g,\kalg)$ is the rank of the symmetric pair.
The invariant theory underlying these constructions was
developed by Kostant--Rallis \cite{KostantRallis}.

Even for two entries, symmetric pairs exhibit phenomena absent
from Richardson's theorem: $C_2(\p)$ can be reducible.
Panyushev proved that, for a symmetric pair of maximal rank
($\ell=\rk\g$), it is an irreducible normal complete intersection
whose defining ideal is generated by the quadratic commuting
equations
\cite[Theorem~3.2, \S3.5, and Corollary~4.4]{Panyushev94}.
Further irreducibility results were obtained in
\cite{Panyushev04,PanyushevYakimova,SabourinYu};
the complete classification is recorded in
\cite[\S4.2 and Table~3]{BuloisIrregular}.
For arbitrary semisimple symmetric pairs, the dimension
$\dim\p+\ell$ and the component
$\overline{K\cdot\mathfrak a^2}$ are classical
\cite[\S3.5 and Proposition~3.7]{Panyushev94};
the uniqueness of this maximal-dimensional component is stated
in \cite[Corollary~2.3]{SabourinYu}; see also \cite{BuloisIrregular}.
In the nilpotent setting, Bulois proved that
$\dim C_2(\N(\p))=\dim\p$ and described its
maximal-dimensional irreducible components
\cite[Proposition~1.5]{Bulois}. 
Invariant-theoretic quotients provide a related direction. Panyushev proved a Chevalley-type restriction theorem in the
maximal-rank case \cite[Theorem~4.5]{Panyushev94}, and
Tevelev established a multivariable restriction theorem for
symmetric pairs \cite{Tevelev}.
Nadimpalli and the author studied the reducedness of higher
commuting-scheme quotients for classical symmetric pairs through
Chevalley restriction \cite{NadimpalliPattanayak}.

Our results concern dimensions and maximal-dimensional
irreducible components for arbitrary tuple length. The first
statement combines \cref{thm:universal,prop:universal-isotropy}.
Let $B$ be the Killing form of $\g$, and identify
$\p^r$ with $\p\otimes\C^r$.
For $J\in\Lambda^2(\C^r)^*$, the symmetric form $B|_\p$
and the alternating form $J$ define a constant alternating
form $B|_\p\otimes J$ on $\p^r$.
Here a subvariety is called isotropic if the form restricts
to zero on its smooth locus.

\begin{theoremA}
For every complex semisimple symmetric pair and every $r\ge2$,
\[
 \dim C_r(\N(\p))
 \le\left\lfloor\frac{r\dim\p}{2}\right\rfloor.
\]
Every irreducible component of $C_r(\N(\p))$ is simultaneously
isotropic for all the forms $B|_\p\otimes J$,
$J\in\Lambda^2(\C^r)^*$.
\end{theoremA}

The dimension bound follows from Bulois's two-variable result
and a coordinate-projection inequality
(\cref{lem:pair-projection}).
The equality case of this inequality supplies additional
geometric information: for $r\ge3$, a component attaining the
half-dimensional bound must be a product of its coordinate
images. This observation is central to the component
classification below.

The bound is attained for Hermitian symmetric pairs. Such a
pair admits a three-grading
\[
 \g=\pmns\oplus\kalg\oplus\pp,\qquad
 \p=\pmns\oplus\pp,
\]
in which $\pp$ and $\pmns$ are abelian subalgebras consisting of
nilpotent elements; see \cref{def:hermitian}.
The following statement collects
\cref{thm:hermitian,thm:hermitian-max-components,%
thm:full-hermitian,cor:full-components}.

\begin{theoremB}
Let $(\g,\kalg)$ be a nonzero simple complex Hermitian symmetric
pair. Put $d=\dim\pp$ and $\ell=\rk(\g,\kalg)$, and choose a
Cartan subspace $\mathfrak a\subseteq\p$.
For every $r\ge2$,
\[
 \dim C_r(\N(\p))=rd,\qquad
 \dim C_r(\p)=\max\{rd,\,2d+(r-1)\ell\}.
\]
The subvarieties $(\pp)^r$ and $(\pmns)^r$ are distinct
maximal-dimensional irreducible components of
$C_r(\N(\p))$ for every $r\ge2$.
For $r\ge3$, they are the only such components.

For every $r\ge2$, the maximal-dimensional irreducible
components of $C_r(\p)$ are precisely those among
\[
 (\pp)^r,\qquad(\pmns)^r,\qquad
 \overline{K\cdot\mathfrak a^r}
\]
whose dimensions attain the displayed maximum. Their
dimensions are $rd$, $rd$, and $2d+(r-1)\ell$, respectively.
\end{theoremB}

For $r=2$, Theorem~B recovers the results of Panyushev and Bulois
recalled above; its new content lies in the range $r\ge3$ and in the
component statements. The distinction between $r=2$ and $r\ge3$ in the nilpotent
component statement is essential: additional maximal-dimensional
components can occur for $r=2$
(\cref{rem:hermitian-r2,ex:hermitian-r2-extra}).
Theorem~B classifies components of maximal dimension; it makes
no assertion that smaller-dimensional components are absent.
For a direct sum of $t$ nonzero simple Hermitian factors, there
are exactly $2^t$ maximal-dimensional nilpotent components when
$r\ge3$ (\cref{cor:hermitian-product-components}).
For even $r\ge2$, every maximal-dimensional nilpotent component
is Lagrangian for each nondegenerate form $B|_\p\otimes J$:
it is isotropic and has half the ambient dimension
(\cref{cor:hermitian-lagrangian}).

The formula for $C_r(\p)$ is obtained by covering it with finitely many
closed families indexed by centralisers of semisimple elements of $\p$,
each bounded by \cref{thm:universal} applied to a smaller symmetric pair. For Hermitian pairs, the restricted-root codimension estimate
gives the dimension bound. Analysis of its equality case then
excludes all intermediate families from maximal dimension
(\cref{lem:hermitian-plevi-codim,lem:intermediate-levi}).
The two remaining families are the nilpotent commuting variety
and the orbit closure of tuples in a Cartan subspace.

A second viewpoint describes the dependence on $r$ through
Grassmannian parameter spaces. Let $\Ab_e(\p)$ be the projective
variety of $e$-dimensional abelian subspaces of $\p$ contained in
$\N(\p)$, and let $\Ab_0(\p)$ be a point. Set
\[
 a(\p)=\max\{e:\Ab_e(\p)\ne\varnothing\},
 \qquad
 m(\p)=\dim\Ab_{a(\p)}(\p),
\]
with the convention $\dim\varnothing=-\infty$.
These spaces are characteristic-zero analogues of the varieties
of elementary subalgebras studied by
Carlson--Friedlander--Pevtsova
\cite[\S1]{CarlsonFriedlanderPevtsova}.
Stratifying a commuting tuple by the dimension of its linear
span gives the following result
(\cref{prop:span-strata,thm:asymptotic}).

\begin{theoremC}
For every complex semisimple symmetric pair and every $r\ge1$,
\[
 \dim C_r(\N(\p))
 =
 \max_{0\le e\le\min\{r,a(\p)\}}
       \bigl(re+\dim\Ab_e(\p)\bigr).
\]
Consequently, for all sufficiently large $r$,
\[
 \dim C_r(\N(\p))=r\,a(\p)+m(\p).
\]
In this range, the maximal-dimensional irreducible components
of $C_r(\N(\p))$ correspond naturally to those of
$\Ab_{a(\p)}(\p)$.
\end{theoremC}

Applied to the symmetric pair
$(\mathfrak h\oplus\mathfrak h,\operatorname{diag}\mathfrak h)$,
this gives the corresponding eventual formula for every complex
semisimple Lie algebra $\mathfrak h$. It provides a
characteristic-zero counterpart of the prediction in
\cite[Introduction]{LevyNgoSivic}.
In positive characteristic, elementary subalgebras consist of
commuting elements annihilated by the restricted $p$-power map;
their parameter spaces therefore concern the restricted
nullcone, which need not equal the full nilpotent cone.
For a nonzero semisimple Hermitian pair, putting
$d_0=\frac12\dim\p$, we obtain
\[
 a(\p)=d_0,\qquad
 \dim\Ab_e(\p)=e(d_0-e)\quad(2\le e\le d_0).
\]
Moreover, the $d_0$-dimensional abelian nilpotent subspaces are
exactly the direct sums obtained by choosing the positive or
negative grading piece in each simple Hermitian factor
(\cref{prop:Ab-hermitian}). In particular $m(\p)=0$, so for Hermitian pairs the eventual formula of
Theorem~C already holds for every $r\ge2$; for $\mathfrak{gl}_n$, by
contrast, it fails for small $r$ (for instance
$\dim C_2(\N(\mathfrak{gl}_4))=15\ne12$).

The geometric applications concern equivariant punctual Quot
schemes and Gorenstein loci. Let the group $\{1,-1\}$ act on
$\mathbb A^r$ by scalar multiplication.
For nonnegative integers $m,n,s_+,s_-$, write
$\mathsf Q^{m,n;s_+,s_-}_{r,0}$ for the reduced scheme
parametrising equivariant quotients
\[
 \mathcal O_{\mathbb A^r}\otimes
       (\C^{s_+}\oplus\C^{s_-})
 \twoheadrightarrow\mathcal F
\]
supported at the origin, where the two framing summands have
characters $+1$ and $-1$, and the corresponding eigenspaces of
$H^0(\mathcal F)$ have dimensions $m$ and $n$.
The higher-tuple part of \cref{thm:equivariant-quot} is as follows.

\begin{theoremD}
Let $r\ge3$, $m,n\ge1$, and $s_+,s_-\ge0$. Then
\[
 \dim\mathsf Q^{m,n;s_+,s_-}_{r,0}
 \le rmn+s_+m+s_-n-m^2-n^2.
\]
Equality holds if and only if at least one of the following
two conditions holds:
\[
 \begin{array}{ll}
 \textup{(i)}&s_-\ge n\ \text{ and }\ m\le s_++rn,\\[1mm]
 \textup{(ii)}&s_+\ge m\ \text{ and }\ n\le s_-+rm.
 \end{array}
\]
When equality holds, the number of maximal-dimensional
irreducible components is the number of these conditions
that hold.
\end{theoremD}

The passage from commuting tuples to Quot schemes uses the
standard presentation by matrices and generating vectors;
compare \cite[Proposition~3.7 and \S3.4]{JelisiejewSivic}.
For $r=2$, the equivariant surface case belongs to the theory of
quiver varieties, where Nakajima's Lagrangian theorem gives
equidimensionality whenever the relevant punctual fibre is
nonempty \cite[Theorem~5.8]{NakajimaALE}; see
\cref{rem:equivariant-quot-surface}.

When $s_+\ge m$ and $s_-\ge n$, both conditions hold, and the two
maximal-dimensional components have a modular description: they are the
loci on which $\mathfrak mH^0(\mathcal F)$ lies in the even, respectively
the odd, part of $H^0(\mathcal F)$.

The second application concerns punctual Gorenstein subschemes. A finite
local algebra $A$ with maximal ideal $\mathfrak m_A$ is Gorenstein when
its socle $\operatorname{Ann}_A(\mathfrak m_A)$ is one-dimensional,
equivalently when it admits a functional $\lambda$ for which
$(u,w)\mapsto\lambda(uw)$ is nondegenerate. The multiplication operators
are self-adjoint for this pairing, so cyclic Gorenstein modules are
governed by the symmetric pair
$(\mathfrak{sl}_N,\mathfrak{so}_N)$ of type $\mathrm{AI}$,
where $N=\dim_\C A$. For invariant Gorenstein quotients,
the socle is either even or odd,
and its parity changes the relevant symmetric pair. For an odd socle, the
pairing twisted by the involution is symplectic and the two
eigenspaces have equal dimension, say $n$. The coordinate
operators are odd with respect to the grading and belong to the
symplectic Lie algebra, giving the pair
$(\mathfrak{sp}_{2n},\mathfrak{gl}_n)$ of type $\mathrm{CI}$.
For an even socle, write $m$ and $n$ for the even and odd dimensions.
Twisting the operators gives the pair
$(\mathfrak{so}_{m+n},\mathfrak{so}_m\oplus\mathfrak{so}_n)$
of type $\mathrm{BDI}$. Let $\mathcal H_{r,N}$ denote the punctual Gorenstein
locus of length $N$ in $\mathbb A^r$, and let
$\mathsf G^{\mathrm{odd}}_{r,2n}$ and $\mathsf G^{\mathrm{even}}_{r,(m,n)}$
denote the invariant loci with odd and even socle and character
multiplicities $(n,n)$ and $(m,n)$, respectively. The following statement
collects \cref{thm:equivariant-gorenstein,cor:odd-socle-strict,%
prop:gorenstein-sharp}.

\begin{theoremE}
Let $r\ge2$. For $N\ge2$,
\[
 \dim\mathcal H_{r,N}\le
 \Bigl\lfloor\frac r2\Bigl(\frac{N(N+1)}2-1\Bigr)\Bigr\rfloor
 -\frac{N(N-1)}2 .
\]
For the invariant loci,
\[
 \dim\mathsf G^{\mathrm{odd}}_{r,2n}\le\frac{rn(n+1)}2-n^2
 \quad(n\ge1),
 \qquad
 \dim\mathsf G^{\mathrm{even}}_{r,(m,n)}\le
 \Bigl\lfloor\frac{rmn}2\Bigr\rfloor-\binom m2-\binom n2
 \quad(m\ge2,\ n\ge1),
\]
and the odd bound improves by one when $r\ge3$ and $n\ge2$.
The even bound is attained for $m=2$ and $1\le n\le r$, but not for
$(m,n)=(3,2)$ and $r\ge3$, although the relevant symmetric pair is
Hermitian in both cases.
\end{theoremE}

The Frobenius pairing is classical, and a presentation of the oriented
Gorenstein locus by commuting self-adjoint matrices is given in
\cite[Proposition~1.7]{JelisiejewIarrobino}. Our counts retain the
Frobenius functional itself and distinguish its two possible characters.
The bounds need not be attained, because the cyclic locus can avoid every
maximal-dimensional component of the ambient commuting variety; the last
assertion of Theorem~E shows that this happens even when the ambient
pair is Hermitian.

The results are formulated over $\C$, and their characteristic-zero
versions extend to every algebraically closed field of
characteristic zero (\cref{prop:charzero}).
For each fixed simple Hermitian root datum, the nilpotent
dimension formula also holds outside finitely many positive
characteristics, with one excluded set valid for all $r\ge2$
(\cref{prop:large-p}).

The paper is organised as follows.
\Cref{sec:prelim} collects preliminaries on symmetric pairs and their
higher commuting schemes, and \cref{sec:span} develops the span
stratification and proves \cref{thm:asymptotic}.
\Cref{sec:universal} proves the universal bound and the isotropy statement.
\Cref{sec:hermitian,sec:full} establish the Hermitian results for the
nilpotent and the full commuting varieties, respectively. Examples and comparisons with known results are given in
\cref{sec:examples}, and the moduli applications in
\cref{sec:applications}. The field of validity is recorded at the
end of \cref{sec:full}, in \cref{subsec:fields}.

\section{Preliminaries on symmetric pairs}
\label{sec:prelim}

This section fixes the Lie-theoretic and scheme-theoretic conventions used
throughout the paper. We recall the decomposition and classification of
complex symmetric pairs, isolate the Hermitian pairs and their
three-gradings, and define the higher commuting schemes.

\subsection{Symmetric pairs and their decomposition}

Let \(\g\) be a complex semisimple Lie algebra and let
\(\theta\in\operatorname{Aut}(\g)\) be an involution. We write
\[
  \g=\kalg\oplus\p,\qquad
  \kalg=\g^\theta,\qquad
\p=\{x\in\g:\theta(x)=-x\}.
\]
Thus
\[
  [\kalg,\kalg]\subseteq\kalg,\qquad
  [\kalg,\p]\subseteq\p,\qquad
  [\p,\p]\subseteq\kalg.
\]
Let \(G\) be the adjoint group of \(\g\), and let \(K\subseteq G\) be the
connected subgroup with Lie algebra \(\kalg\). The action of \(K\) on
\(\p\) is the isotropy representation of the symmetric pair
\((\g,\kalg)\). We call the pair \emph{nonzero} if \(\p\neq0\), and of
\emph{simple type} if \(\g\) is simple.

Let
\[
  B(x,y)=\operatorname{tr}(\operatorname{ad}x\,\operatorname{ad}y)
\]
be the Killing form of \(\g\). From the \(\theta\)-invariance and
nondegeneracy of \(B\), the decomposition \(\g=\kalg\oplus\p\) is
\(B\)-orthogonal, and the restrictions of \(B\) to \(\kalg\) and to \(\p\)
are nondegenerate.

For a subset $S\subseteq\g$ and a vector subspace $U\subseteq\g$,
we write
\[
 \mathfrak z_U(S)=\{u\in U:[u,s]=0\text{ for every }s\in S\}.
\]
When $U=\mathfrak h$ is a Lie subalgebra, this is its centralizer
of $S$; its center is $\mathfrak z(\mathfrak h)
=\mathfrak z_{\mathfrak h}(\mathfrak h)$.
For an algebraic subgroup $H\subseteq G$, let $Z_H(S)$ and
$N_H(S)$ denote the pointwise centralizer and setwise normalizer.
We also write
\[
 \mathfrak n_\g(\mathfrak l)
 =\{x\in\g:[x,\mathfrak l]\subseteq\mathfrak l\}
\]
for the Lie algebra normalizer of a subalgebra $\mathfrak l$.

The standard decomposition below reduces questions for semisimple symmetric pairs to two classes of indecomposable factors; see
\cite[\S1.1]{Bulois} and \cite[Chapter~V]{Helgason}.

\begin{proposition}\label{prop:pair-decomposition}
Every complex semisimple symmetric pair is a direct sum of indecomposable
symmetric pairs of the following two kinds:
\begin{enumerate}[label=\textup{(\roman*)}]
\item a pair of simple type, obtained by restricting \(\theta\) to a
  \(\theta\)-stable simple ideal of \(\g\);
\item a pair of group type
  \[
    (\mathfrak h \oplus\mathfrak h,\Delta\mathfrak h),
    \qquad \Delta\mathfrak h:=\{(x,x):x\in\mathfrak h\},
  \]
  where \(\mathfrak h\) is simple and the involution interchanges the two
  factors.
\end{enumerate}
\end{proposition}

An element \(x\in\g\) is called \emph{nilpotent} if \(\operatorname{ad}x\) is
nilpotent. Let
\[
  \N(\g)\subseteq\g,\qquad
  \N(\p):=\N(\g)\cap\p
\]
be the nilpotent cones. A \emph{Cartan subspace} is a maximal abelian
subspace \(\mathfrak a\subseteq\p\) consisting of semisimple elements. Any
two Cartan subspaces in \(\p\) are conjugate under \(K\), and
\[
  \rk(\g,\kalg):=\dim\mathfrak a
\]
is the rank of the symmetric pair. By the Chevalley restriction theorem for
symmetric pairs and the results of Kostant--Rallis \cite{KostantRallis},
\(\N(\p)\) is the null cone of the \(K\)-action on \(\p\), is
equidimensional, and
\begin{equation}\label{eq:KR}
  \dim\N(\p)=\dim\p-\rk(\g,\kalg).
\end{equation}

\subsection{Hermitian symmetric pairs and their classification}
\label{subsec:hermitian-classification}

After discarding factors with zero $(-1)$-eigenspace, every
semisimple Hermitian pair is a direct sum of nonzero Hermitian
pairs of simple type. A choice of three-grading supplies a
polarisation of the isotropy representation into two abelian
nilpotent subspaces.

\begin{definition}\label{def:hermitian}
A symmetric pair \((\g,\kalg)\) is of \emph{Hermitian type} if \(\g\) admits
a three-grading
\[
  \g=\g_{-1}\oplus\g_0\oplus\g_1,
  \qquad
  [\g_i,\g_j]\subseteq\g_{i+j},
  \qquad
  \g_i=0\ \text{for }|i|>1,
\]
such that
\[
  \g_0=\kalg,\qquad
  \p=\g_{-1}\oplus\g_1.
\]
We write
\[
  \pmns:=\g_{-1},\qquad \pp:=\g_1.
\]
\end{definition}

The grading is induced by a semisimple grading element
\(\xi\in\mathfrak{z}(\kalg)\) satisfying
\[
  [\xi,x]=j x\qquad(x\in\g_j).
\]
For a nonzero pair with \(\g\) simple, \((\g,\kalg)\) is Hermitian if and only
if \(\mathfrak{z}(\kalg)\neq0\). In that case \(\dim\mathfrak{z}(\kalg)=1\),
and its generator may be normalised so that \(\operatorname{ad}\xi\) has
eigenvalues \(-1,0,1\). These are the complexifications of the irreducible
Hermitian symmetric spaces of noncompact type; see
\cite[Chapter~VIII, \S\S5--6]{Helgason}.

The nonzero Hermitian symmetric pairs of simple type are listed
below; see also \cite[Chapter~X, \S6.3]{Helgason}.
In each case, $\mathfrak z(\kalg)=\C\xi$.
Replacing $\xi$ by $-\xi$ exchanges $\pp$ and $\pmns$.

\begin{center}
\small
\begin{tabular}{@{}llll@{}}
\toprule
Type & \((\g,\kalg)\) & \(\pp\) as a \([\kalg,\kalg]\)-module & \(\dim\pp\)\\
\midrule
\(\mathrm{AIII}\)
& \(\bigl(\mathfrak{sl}_{m+n},
    \mathfrak{s}(\mathfrak{gl}_m\oplus\mathfrak{gl}_n)\bigr)\)
& \(\Hom(\C^n,\C^m)\)
& \(mn\)\\[1mm]
\(\mathrm{BDI}_{\mathrm{Herm}}\)
& \(\bigl(\mathfrak{so}_{n+2},
    \mathfrak{so}_n\oplus\mathfrak{so}_2\bigr)\)
& standard \(\mathfrak{so}_n\)-module
& \(n\)\\[1mm]
\(\mathrm{CI}\)
& \(\bigl(\mathfrak{sp}_{2n},\mathfrak{gl}_n\bigr)\)
& \(\Sym^2\C^n\)
& \(\frac{n(n+1)}2\)\\[1mm]
\(\mathrm{DIII}\)
& \(\bigl(\mathfrak{so}_{2n},\mathfrak{gl}_n\bigr)\)
& \(\Lambda^2\C^n\)
& \(\frac{n(n-1)}2\)\\[1mm]
\(\mathrm{EIII}\)
& \(\bigl(\mathfrak e_6,\mathfrak{so}_{10}\oplus\C\bigr)\)
& half-spin module \(\Delta_{10}^{+}\)
& \(16\)\\[1mm]
\(\mathrm{EVII}\)
& \(\bigl(\mathfrak e_7,\mathfrak e_6\oplus\C\bigr)\)
& minuscule \(27\)-dimensional module
& \(27\)\\
\bottomrule
\end{tabular}
\end{center}

Here $\mathfrak s(\mathfrak{gl}_m\oplus\mathfrak{gl}_n)$
denotes the block-diagonal subalgebra on which the sum of the
two traces is zero. We require $m,n\ge1$ in type $\mathrm{AIII}$,
$n\ge3$ in type $\mathrm{BDI}_{\mathrm{Herm}}$, $n\ge1$ in
type $\mathrm{CI}$, and $n\ge3$ in type $\mathrm{DIII}$.
The notation $\Delta_{10}^+$ denotes one of the two half-spin
representations of $\mathfrak{so}_{10}$.
Low-rank coincidences include
$\mathrm{DIII}(3)\simeq\mathrm{AIII}(1,3)$ and
$\mathrm{BDI}_{\mathrm{Herm}}(3)\simeq\mathrm{CI}(2)$.

The following are the standard properties of the three-grading attached to a
Hermitian symmetric pair; see
\cite[Chapter~VIII, \S\S5--6]{Helgason}.  The first five assertions follow from
the degree rule, the vanishing of $\g_j$ for $|j|>1$, and the
invariance and nondegeneracy of the Killing form. The final
assertion is the standard irreducibility property in the simple
Hermitian classification displayed above.

\begin{lemma}\label{lem:hermitian-basic}
Let \(\g=\pmns\oplus\kalg\oplus\pp\) be a Hermitian symmetric pair. Then:
\begin{enumerate}[label=\textup{(\roman*)}]
\item \(\pp\) and \(\pmns\) are abelian subalgebras of \(\g\);
\item every element of \(\pp\cup\pmns\) is nilpotent;
\item \(B(\pp,\pp)=B(\pmns,\pmns)=0\) and \(B(\kalg,\pp)=B(\kalg,\pmns)=0\);
\item \(B\) induces a perfect pairing \(\pp\times\pmns\longrightarrow\C\);
\item \(\dim\pp=\dim\pmns=\frac12\dim\p\).
\item If $\g$ is simple and $\p\ne0$, then $\pp$ and $\pmns$
are irreducible $K$-modules; the grading element $\xi$ acts on
them with eigenvalues $+1$ and $-1$, respectively.
\end{enumerate}
\end{lemma}

Thus every Hermitian pair admits a \(K\)-stable polarisation
\[
  \p=\pmns\oplus\pp
\]
into complementary \(B\)-isotropic, abelian, nilpotent subspaces. In
particular,
\[
  (\pp)^r\cup(\pmns)^r\subseteq\Cr(\N(\p))
\]
for every \(r\geq1\).

\subsection{Higher commuting schemes}
\label{subsec:commuting-schemes} We now define the commuting schemes studied in this paper, both with and
without the nilpotency condition, and record their elementary symmetries. Let
\[
    \chi_{\p}\colon\p\longrightarrow
  \p/\!/K:=\Spec\C[\p]^K
\]
be the invariant-theoretic quotient, and let
\[
  \mathcal N_{\mathrm{sch}}(\p)
  :=\chi_{\p}^{-1}\bigl(\chi_{\p}(0)\bigr)
\]
denote its scheme-theoretic null fibre. Its underlying reduced variety is
\(\N(\p)\).

\begin{definition}\label{def:Cr}
For \(r\geq1\), let \(\mathcal C_r(\p)\) be the affine scheme in \(\p^r\)
defined by the equations
\[
  [x_i,x_j]=0\qquad(1\leq i<j\leq r),
\]
and define the nilpotent higher commuting scheme by
\[
  \mathcal C_r^{\mathrm{nil}}(\p)
  :=
  \mathcal C_r(\p)
  \times_{\p^r}
  \mathcal N_{\mathrm{sch}}(\p)^r.
\]
We write
\[
  \Cr(\p):=\mathcal C_r(\p)_{\mathrm{red}},
  \qquad
  \Cr(\N(\p))
  :=\mathcal C_r^{\mathrm{nil}}(\p)_{\mathrm{red}}.
\]
Equivalently,
\[
  \Cr(\N(\p))
  =
  \left\{
    (x_1,\ldots,x_r)\in\N(\p)^r
    \ \middle|\
    [x_i,x_j]=0\ \text{for all }i,j
  \right\}.
\]
\end{definition}

The distinction between the scheme and its reduction is harmless for
dimension and irreducible-component questions. The results below do not
assert reducedness of the unreduced commuting schemes.

\begin{lemma}\label{lem:symmetries}
The variety \(\Cr(\N(\p))\) is a closed cone in \(\p^r\). It is stable under
the diagonal action of \(K\) and under the action of \(\GL_r(\C)\) given by
\[
  g\cdot(x_1,\ldots,x_r)
  =
  \left(
    \sum_jg_{1j}x_j,\ldots,\sum_jg_{rj}x_j
  \right).
\]
\end{lemma}

\begin{proof}
The closedness and \(K\)-stability follow directly from
\cref{def:Cr}. For a point
\((x_1,\ldots,x_r)\in\Cr(\N(\p))\), we put
\[
  A=\operatorname{span}\{x_1,\ldots,x_r\}\subseteq\p.
\]
The endomorphisms
\(\operatorname{ad}x_1,\ldots,\operatorname{ad}x_r\) commute and are
nilpotent. They are therefore simultaneously strictly
upper triangular in a suitable basis.  Every linear combination of them is
therefore nilpotent.  It follows that every element of \(A\) is nilpotent, so
every \(\GL_r(\C)\)-transform of the tuple again lies in
\(\Cr(\N(\p))\).  Finally, restricting this action to the subgroup of scalar matrices 
\(\C^\times I_r \subseteq \GL_r(\C)\) shows that \(\Cr(\N(\p))\) is stable under 
simultaneous dilation, and is therefore a cone. 
\end{proof}

\begin{remark}\label{rem:GLr-symmetry}
In the case \(r=2\), this \(\GL_2(\C)\)-symmetry was utilized by Bulois 
\cite[proof of Lemma~1.8]{Bulois}. For general \(r\), one may canonically identify 
\(\p^r\) with the tensor product \(\p\otimes U\), where \(U = \C^r\). Under this 
identification, the group \(\GL_r(\C) \cong \GL(U)\) acts naturally on the second 
tensor factor. For each alternating form \(J \in \Lambda^2 U^*\), the bilinear form
\[
  \omega_J := B|_\p \otimes J
\]
defines an alternating 2-form on \(\p\otimes U\) satisfying the equivariance property 
\(g^*\omega_J = \omega_{g^*J}\) for every \(g \in \GL(U)\). Because \(\GL(U)\) acts 
transitively on the set of alternating forms on \(U\) of any given rank, this 
symmetry clarifies why the subsequent isotropy properties hold simultaneously 
for the entire family \(\{\omega_J : J \in \Lambda^2 U^*\}\).
\end{remark}
Commuting schemes are compatible with direct sums of symmetric pairs, so
most questions reduce to indecomposable pairs.
\begin{lemma}\label{lem:products}
Suppose
\[
  (\g,\kalg)
  =
  (\g^{(1)},\kalg^{(1)})
  \oplus
  (\g^{(2)},\kalg^{(2)}),
  \qquad
  \p=\p^{(1)}\oplus\p^{(2)}.
\]
Then
\[
  \mathcal C_r^{\mathrm{nil}}(\p)
  \simeq
  \mathcal C_r^{\mathrm{nil}}(\p^{(1)})
  \times
  \mathcal C_r^{\mathrm{nil}}(\p^{(2)}).
\]
Consequently,
\[
  \Cr(\N(\p))
  \simeq
  \Cr(\N(\p^{(1)}))
  \times
  \Cr(\N(\p^{(2)})),
\]
and dimensions add. The same product decomposition holds for the full commuting schemes:
\[
 \mathcal C_r(\p)\cong
 \mathcal C_r(\p^{(1)})\times\mathcal C_r(\p^{(2)}),
 \qquad
 C_r(\p)\cong C_r(\p^{(1)})\times C_r(\p^{(2)}).
\]
\end{lemma}

\begin{proof}
We write $x_i=x_i^{(1)}+x_i^{(2)}$. Since the two summands
are commuting ideals, the bracket equations split into the
independent equations $[x_i^{(1)},x_j^{(1)}]=0$ and
$[x_i^{(2)},x_j^{(2)}]=0$. This proves the assertion for the
full commuting schemes.

Let $K^{(j)}$ be the connected fixed-point group for the $j$th
factor. The product decomposition of the action gives
\[
 \C[\p]^{K^{(1)}\times K^{(2)}}
 =\C[\p^{(1)}]^{K^{(1)}}\otimes_\C
   \C[\p^{(2)}]^{K^{(2)}}.
\]
Hence the invariant quotient morphism is the product of the
two quotient morphisms. Taking the fibre over the image of
zero shows that the null schemes also factor as products.
Intersecting these null-fibre conditions with the bracket
equations proves the nilpotent scheme isomorphism. Over an algebraically closed field of characteristic zero, the Cartesian product of 
any two reduced schemes of finite type is again reduced. Passing to the underlying 
reduced subschemes yields the isomorphisms \(\Cr(\N(\p)) \cong \Cr(\N(\p^{(1)})) \times \Cr(\N(\p^{(2)}))\) 
and \(C_r(\p) \cong C_r(\p^{(1)}) \times C_r(\p^{(2)})\). The additivity of Krull 
dimensions follows immediately.
\end{proof}

\begin{example}\label{ex:groupcase}
Let \(\mathfrak h\) be complex semisimple and let \(\theta\) interchange the
two factors of \(\g=\mathfrak h\oplus\mathfrak h\). Then
\[
  \kalg=\Delta\mathfrak h,\qquad
  \p=\{(x,-x):x\in\mathfrak h\}\simeq\mathfrak h.
\]
Under this identification,
\[
  [(x,-x),(y,-y)]=([x,y],[x,y]),
\]
so two elements of \(\p\) commute precisely when the corresponding elements
of \(\mathfrak h\) commute. Moreover, \((x,-x)\) is nilpotent if and only if
\(x\) is nilpotent. Therefore
\[
  \Cr(\N(\p))\simeq\Cr(\N(\mathfrak h)),
  \qquad
  \dim\p=\dim\mathfrak h.
\]
The restriction of the Killing form of \(\g\) to \(\p\) is
\(2B_{\mathfrak h}\). Consequently, general results established for complex semisimple symmetric pairs 
specialize directly to statements concerning the classical nilpotent higher commuting 
varieties of semisimple Lie algebras. Note that a nonzero group-type pair is never Hermitian, because
\(\kalg\simeq\mathfrak h\) is semisimple and has trivial center.
\end{example}

\section{Span stratification and eventual dimension}
\label{sec:span}

We parametrise commuting nilpotent tuples by their linear spans. This gives
an exact formula for $\dim\Cr(\N(\p))$ in terms of parameter spaces of
abelian nilpotent subspaces, and its eventual linear growth in $r$.

\subsection{Abelian nilpotent subspaces}
\label{subsec:span}
We first introduce the parameter spaces of abelian nilpotent subspaces
and show that they are projective varieties.

For a complex vector space \(V\) and an integer \(d\ge0\), we denote by \(\Gr(d,V)\) 
the Grassmannian of \(d\)-dimensional linear subspaces of \(V\), and we abbreviate 
\(\Gr(d,N) := \Gr(d,\C^N)\). For \(1\le d\le\dim\p\), we define the parameter space 
of \(d\)-dimensional abelian nilpotent subspaces of \(\p\) by
\[
  \Ab_d(\p)
  :=
  \left\{
    A\in\Gr(d,\p)
    \ \middle|\
    [A,A]=0\ \text{ and }\ A\subseteq\N(\p)
  \right\},
\]
with the conventions that \(\Ab_0(\p) := \Spec\C\) and \(\Ab_d(\p) := \varnothing\) 
for all \(d>\dim\p\). We equip each parameter space \(\Ab_d(\p)\) with its reduced 
subscheme structure, and we set \(\dim\varnothing = -\infty\).

The parameter space \(\Ab_d(\p)\) is a projective algebraic variety.
The construction is analogous to that for elementary subalgebras in
\cite[\S1]{CarlsonFriedlanderPevtsova}. Varieties of elementary subalgebras of maximal dimension were studied
further in \cite{PevtsovaStark}.
To verify closedness directly, let 
\(\mathcal S_d \subseteq \p\otimes\mathcal O_{\Gr(d,\p)}\) denote the tautological 
subbundle of rank \(d\), with \(\mathcal S_0 = 0\). The Lie bracket 
\([\cdot,\cdot]\colon\Lambda^2\p\to\kalg\) restricts on each fibre to the commutator, 
inducing a homomorphism of vector bundles
\[
  [\cdot,\cdot]\colon
  \Lambda^2\mathcal S_d
  \longrightarrow
  \kalg\otimes\mathcal O_{\Gr(d,\p)}.
\]
The commutativity condition \([A,A]=0\) is precisely the zero locus of this bundle 
morphism, which is Zariski-closed in \(\Gr(d,\p)\). 

For the nilpotence condition, let \(f_1,\ldots,f_q \in \C[\p]\) be homogeneous 
polynomial generators of the defining ideal of the nilpotent cone \(\N(\p)\), 
with degrees \(n_j := \deg(f_j)\). Restricting these polynomials to the tautological 
subspaces defines global sections
\[
  f_j|_{\mathcal S_d} \in H^0\bigl(\Gr(d,\p),\,\Sym^{n_j}\mathcal S_d^*\bigr)
  \qquad (1\le j\le q).
\]
A subspace \(A\in\Gr(d,\p)\) is contained in \(\N(\p)\) if and only if each section 
\(f_j|_{\mathcal S_d}\) vanishes at \(A\). Let $Z_d$ be the scheme-theoretic intersection of their zero loci
with the zero locus of the bracket map. Then
\[
  \Ab_d(\p)=(Z_d)_{\mathrm{red}}.
\]
In particular, $\Ab_d(\p)$ is closed in $\Gr(d,\p)$ and hence
projective.

We associate with the symmetric pair \((\g,\kalg)\) the numerical invariants
\[
  a(\p) := \max\bigl\{d\in\Z_{\ge0} : \Ab_d(\p)\neq\varnothing\bigr\},
  \qquad
  m(\p) := \dim\Ab_{a(\p)}(\p).
\]
Because any linear subspace of an abelian nilpotent subspace remains abelian and 
nilpotent, every \(A\in\Ab_{a(\p)}(\p)\) contains subspaces of each dimension 
\(d\le a(\p)\). It follows that
\[
  \Ab_d(\p) \neq \varnothing
  \qquad\text{for all } 0 \le d \le a(\p).
\]

\subsection{The span strata and the exact dimension formula} We now stratify $\Cr(\N(\p))$ by the dimension of the span of a tuple;
each stratum is an open subset of a vector bundle over $\Ab_d(\p)$. For \(0\leq d\leq\min\{r,\dim\p\}\), let
\[
  \Cr(\N(\p))^{(d)}
  :=
  \left\{
    \mathbf x=(x_1,\ldots,x_r)\in\Cr(\N(\p))
    \ \middle|\
    \dim\langle x_1,\ldots,x_r\rangle=d
  \right\}.
\]
These locally closed subsets form the stratification of
\(\Cr(\N(\p))\) by the rank of the span map.

\begin{proposition}\label{prop:span-strata}
Let \(r\geq1\) and \(0\leq d\leq\min\{r,\dim\p\}\). There is a natural isomorphism
\begin{equation}\label{eq:span-bundle}
  \Cr(\N(\p))^{(d)}
  \simeq
  \left\{
    (A,x_1,\ldots,x_r)
    \in
    \left.\mathcal S_d^{\oplus r}\right|_{\Ab_d(\p)}
    \ \middle|\
    \langle x_1,\ldots,x_r\rangle=A
  \right\}.
\end{equation}
The right-hand side is the fibrewise spanning open subset of the vector
bundle \(\left.\mathcal S_d^{\oplus r}\right|_{\Ab_d(\p)}\). Consequently,
whenever \(\Ab_d(\p)\neq\varnothing\),
\begin{equation}\label{eq:span-stratum-dim}
  \dim\Cr(\N(\p))^{(d)}
  =
  rd+\dim\Ab_d(\p).
\end{equation}

Moreover,
\[
  \bigcup_{e=0}^{d}\Cr(\N(\p))^{(e)}
\]
is the closed rank-at-most-\(d\) locus. Hence \(\Cr(\N(\p))^{(d)}\) is locally
closed and is open in the rank-at-most-\(d\) locus. In particular, the stratum
of maximal possible rank \(\min\{r,a(\p)\}\) is open in \(\Cr(\N(\p))\).
Finally,
\begin{equation}\label{eq:span-dimension}
  \dim\Cr(\N(\p))
  =
  \max_{0\leq d\leq\min\{r,a(\p)\}}
  \left(
    rd+\dim\Ab_d(\p)
  \right).
\end{equation}
\end{proposition}

\begin{proof}
Let \(\mathbf x=(x_1,\ldots,x_r)\in\Cr(\N(\p))^{(d)}\), and put
\[
  A_{\mathbf x}:=\langle x_1,\ldots,x_r\rangle.
\]
The operators \(\operatorname{ad}x_1,\ldots,\operatorname{ad}x_r\) are
commuting nilpotent endomorphisms and hence are simultaneously strictly
triangularisable. Every element of
\(A_{\mathbf x}\) is therefore nilpotent. Since the \(x_i\) commute, this
proves \(A_{\mathbf x}\in\Ab_d(\p)\). Thus the map
\[
  \mathbf x
  \longmapsto
  (A_{\mathbf x},x_1,\ldots,x_r)
\]
has the target displayed in \eqref{eq:span-bundle}.

On the rank-$d$ locus in $\Hom(\C^r,\p)$, the map
$f\mapsto\operatorname{im}f$ is a morphism to $\Gr(d,\p)$.
Indeed, on a chart where a fixed $d\times d$ minor is
nonzero, its Grassmannian coordinates are ratios of minors
with that nonvanishing denominator. For $d=0$, the map is
constant. Applying this to the linear map $e_i\mapsto x_i$,
where $(e_i)$ is the standard basis of $\C^r$, shows that
$\mathbf x\mapsto(A_{\mathbf x},\mathbf x)$ is a morphism.
Its inverse is the projection $(A,\mathbf x)\mapsto\mathbf x$,
which proves \eqref{eq:span-bundle}.

The condition that \(r\) vectors span a \(d\)-dimensional fibre is open and,
since \(d\leq r\), nonempty. Over every irreducible component of
\(\Ab_d(\p)\), the right-hand side of \eqref{eq:span-bundle} is therefore a
dense open subset of a vector bundle of rank \(rd\). This proves
\eqref{eq:span-stratum-dim}.

The rank of the linear map
\[
  \C^r\longrightarrow\p,\qquad e_i\longmapsto x_i,
\]
is at most \(d\) precisely when all \((d+1)\times(d+1)\) minors vanish.
Thus the rank-at-most-\(d\) locus is closed. The remaining assertions follow
from the finite disjoint decomposition
\[
  \Cr(\N(\p))
  =
  \bigsqcup_{0\leq d\leq\min\{r,a(\p)\}}
  \Cr(\N(\p))^{(d)}.
\]
\end{proof}

\subsection{Eventual dimension and maximal components}

The exact maximum in \eqref{eq:span-dimension} is eventually attained only at
the largest possible span. The analogous asymptotic mechanism for elementary
subalgebras is discussed in \cite{LevyNgoSivic}. We record the
statement here because the explicit threshold and the correspondence of
maximal-dimensional components are needed later.

\begin{theorem}
\label{thm:asymptotic}
Set $a:=a(\p)$ and $m:=\dim\Ab_a(\p)$.
With the convention \(\dim\Ab_0(\p)=0\), there is an integer
\(r_0\geq\max\{1,a\}\) such that
\begin{equation}\label{eq:eventual-dimension}
  \dim\Cr(\N(\p))=ra+m
  \qquad(r\geq r_0).
\end{equation}
One valid choice is
\begin{equation}\label{eq:explicit-r0}
  r_0
  :=
  \max\left(
    \{1,a\}
    \cup
    \left\{
      1+\left\lfloor
        \frac{\dim\Ab_d(\p)-m}{a-d}
      \right\rfloor
      :0\leq d<a
    \right\}
  \right),
\end{equation}
where the second set is omitted when \(a=0\).

For every \(r\geq r_0\), the maximal-dimensional irreducible components of
\(\Cr(\N(\p))\) are in natural bijection with the irreducible components
\(T\) of \(\Ab_a(\p)\) satisfying \(\dim T=m\). The component associated
with \(T\) is
\[
  \overline{
    \left\{
      (x_1,\ldots,x_r)\in A^r
      \ \middle|\
      A\in T,\ 
      \langle x_1,\ldots,x_r\rangle=A
    \right\}
  }.
\]
\end{theorem}

\begin{proof}
If $a=0$, then $\N(\p)=\{0\}$: any nonzero nilpotent
element would span a member of $\Ab_1(\p)$.
Thus $\Cr(\N(\p))$ and $\Ab_0(\p)$ are points, $m=0$,
and all assertions hold with $r_0=1$.
We may therefore assume $a>0$.

If \(0\leq d<a\), then the definition of \(r_0\) implies
\[
  r > \frac{\dim\Ab_d(\p)-m}{a-d}.
\]
Equivalently,
\[
  ra+m > rd+\dim\Ab_d(\p).
\]
Since \(r\geq a\), all ranks \(0,\ldots,a\) occur in the maximum
\eqref{eq:span-dimension}. The preceding strict inequalities show that the
maximum is uniquely achieved at \(d=a\), proving
\eqref{eq:eventual-dimension}.

Set \(U_a:=\Cr(\N(\p))^{(a)}\). By \cref{prop:span-strata}, \(U_a\) is open in
\(\Cr(\N(\p))\), and it is the fibrewise spanning open subset of
\(\left.\mathcal S_a^{\oplus r}\right|_{\Ab_a(\p)}\). The complement of
\(U_a\) is the closed locus \(\bigcup_{d<a}\Cr(\N(\p))^{(d)}\), which has
dimension
\[
  \max_{0\le d<a}\dim\Cr(\N(\p))^{(d)} < ra+m.
\]
Therefore, every irreducible component of \(\Cr(\N(\p))\) of maximal dimension
\(ra+m\) must intersect \(U_a\), and no maximal component is contained in the
boundary.

Since \(U_a\to\Ab_a(\p)\) is a fibre bundle with irreducible fibres of
constant dimension \(ra\), its irreducible components are indexed by the
irreducible components of the base \(\Ab_a(\p)\). Specifically, if
\(T\subseteq\Ab_a(\p)\) is an irreducible component, the preimage of \(T\) in
\(U_a\) is irreducible of dimension \(ra+\dim T\). It follows that the
components of \(U_a\) of maximal dimension correspond precisely to the
components \(T\subseteq\Ab_a(\p)\) with \(\dim T=m\). Taking closures yields
the asserted bijection and description of components.
\end{proof}

\begin{remark}\label{rem:eventual-meaning}
The theorem identifies two intrinsic geometric invariants of the symmetric pair:
\[
  a(\p)
  =
  \max\{\dim A:A\subseteq\N(\p),\ [A,A]=0\}
\]
is the eventual slope, while \(\dim\Ab_{a(\p)}(\p)\) is the eventual intercept.
Thus the asymptotic problem separates into the linear-algebraic problem of
finding the largest abelian nilpotent subspaces and the geometric problem of
determining their parameter variety.
\end{remark}

\section{Universal bounds and simultaneous isotropy}
\label{sec:universal}

The dimension estimate follows from the two-variable theorem and an
elementary inequality for coordinate projections. Its equality case will
also determine the maximal-dimensional components in the Hermitian case.

\begin{lemma}\label{lem:pair-projection}
Let $E$ be a finite-dimensional vector space over an algebraically closed
field $k$, let $r\ge2$, and let $Z\subseteq E^r$ be a nonempty irreducible
closed subvariety. Write $\pi_i$ and $\pi_{ij}$ for the projections onto
the indicated factors. Then
\[
 (r-1)\dim Z
 \le \sum_{1\le i<j\le r}\dim\overline{\pi_{ij}(Z)}.
\]
If $r\ge3$, $d\ge0$ is an integer, $\dim Z=rd$, and
$\dim\overline{\pi_{ij}(Z)}\le2d$ for all $i<j$, then
\[
 Z=Y_1\times\cdots\times Y_r,
 \qquad Y_i=\overline{\pi_i(Z)},\qquad \dim Y_i=d.
\]
\end{lemma}

\begin{proof}
We choose linear coordinates on each factor of $E^r$.
Their restrictions generate the function field $k(Z)$ over
the ground field $k$. We choose a transcendence basis from
these restrictions and let $b_i$ be the number chosen from
the $i$th factor. Then
\[
 \dim Z=\sum_{i=1}^r b_i.
\]
The coordinates chosen from factors $i$ and $j$ are
algebraically independent on $\overline{\pi_{ij}(Z)}$.
Therefore
\[
 b_i+b_j\le\dim\overline{\pi_{ij}(Z)}.
\]
Summing over all pairs proves the inequality, since each
$b_i$ appears exactly $r-1$ times.

Under the additional hypotheses, we obtain
\[
 (r-1)rd
 =\sum_{i<j}(b_i+b_j)
 \le\sum_{i<j}\dim\overline{\pi_{ij}(Z)}
 \le\binom r2(2d)
 =(r-1)rd.
\]
Hence every pair satisfies $b_i+b_j=2d$.
Since $r\ge3$, these equations imply $b_i=d$ for every $i$.
This conclusion holds for every transcendence basis selected
from the coordinate functions.

Fix $i$, and put $Y_i=\overline{\pi_i(Z)}$.
We first select a transcendence basis of $k(Y_i)$ from the
coordinates on the $i$th factor. Dominance of $Z\to Y_i$
makes these functions algebraically independent in $k(Z)$,
so we can extend them to a transcendence basis of $k(Z)$
using the remaining coordinate functions. No additional
coordinate from the $i$th factor can be selected: all such
coordinates are already algebraic over the chosen basis of
$k(Y_i)$. Thus this choice has $b_i=\dim Y_i$, and the
preceding argument gives $\dim Y_i=d$.

Finally, $Z$ is a closed irreducible subvariety of
$Y_1\times\cdots\times Y_r$. The product is irreducible
over the algebraically closed ground field, and both varieties
have dimension $rd$. Consequently, they are equal.
\end{proof}

We use the following known result.

\begin{proposition}[Bulois {\cite[Proposition~1.5]{Bulois}}]
\label{prop:r2-nilpotent}
For a complex semisimple symmetric pair $\g=\kalg\oplus\p$,
\[
 \dim C_2(\N(\p))=\dim\p.
\]
\end{proposition}

\begin{theorem}\label{thm:universal}
Let $\g=\kalg\oplus\p$ be a complex semisimple symmetric pair.
For every $r\ge2$,
\[
 \dim\Cr(\N(\p))
 \le\left\lfloor\frac{r\dim\p}{2}\right\rfloor.
\]
\end{theorem}

\begin{proof}
For every irreducible component $Z$ of $\Cr(\N(\p))$, each
$\overline{\pi_{ij}(Z)}$ is contained in $C_2(\N(\p))$.
Consequently, \cref{lem:pair-projection,prop:r2-nilpotent} give
\[
 (r-1)\dim Z\le\binom r2\dim\p.
\]
Dividing by $r-1$ and using the integrality of $\dim Z$ proves the assertion.
\end{proof}

We next record the associated isotropy statement. Let $B$ be the
Killing form of $\g$. Its restriction to $\p$ is nondegenerate.
For $1\le i<j\le r$, define the constant alternating form
\[
 \omega_{ij}(\mathbf u,\mathbf v)
 =B(u_i,v_j)-B(u_j,v_i)
 \qquad(\mathbf u,\mathbf v\in\p^r).
\]
These forms span the family $B|_\p\otimes J$, where
$J\in\Lambda^2(\C^r)^*$. We write $Z_{\sm}$ for the smooth locus
of a reduced variety $Z$. Isotropy means that the form restricts
to zero on this locus. For a symplectic form, a subvariety is
Lagrangian if it is isotropic and has half the ambient dimension.

\begin{proposition}
\label{prop:universal-isotropy}
Let $Z$ be an irreducible component of $\Cr(\N(\p))$.
Every form $\omega_{ij}$ vanishes on the smooth locus of $Z$.
Equivalently, $Z$ is isotropic for every form $B|_\p\otimes J$.
\end{proposition}

\begin{proof}
Write $x_i:Z_{\sm}\to\N(\p)$ for the $i$th coordinate map,
and consider the regular one-form
\[
  \alpha_{ij}=B(x_j,dx_i)\in
             H^0(Z_{\sm},\Omega^1_{Z_{\sm}}).
\]
There are finitely many nilpotent $K$-orbits in $\p$
\cite{KostantRallis}; see also \cite[\S1]{Bulois}.
Their inverse images under $x_i$ give a finite partition of
$Z_{\sm}$ into locally closed subsets. Since $Z_{\sm}$ is
irreducible, one such subset contains a dense open subset
$U\subseteq Z_{\sm}$. Thus $x_i(U)$ lies in a single orbit.

For $z\in U$ and $u\in T_zZ$, the differential satisfies
\[
  (dx_i)_z(u)\in T_{x_i(z)}(K\cdot x_i(z))
              =[\kalg,x_i(z)].
\]
We may therefore write $(dx_i)_z(u)=[a,x_i(z)]$ for some
$a\in\kalg$. Invariance of $B$ and commutativity give
\[
  \alpha_{ij,z}(u)
  =B(x_j(z),[a,x_i(z)])
  =B([x_i(z),x_j(z)],a)=0.
\]
Hence $\alpha_{ij}$ vanishes on $U$. The sheaf
$\Omega^1_{Z_{\sm}}$ is locally free on the integral variety
$Z_{\sm}$, so a section vanishing on a dense open subset
vanishes everywhere. Thus $\alpha_{ij}=0$ on $Z_{\sm}$.

The one-form $\alpha_{ij}$ is the pullback of the ambient
one-form $B(x_j,dx_i)$ on $\p^r$, whose exterior derivative
is $-\omega_{ij}$  (with the convention $(\varphi\wedge\psi)(u,v)=\varphi(u)\psi(v)-\varphi(v)\psi(u)$). Since exterior differentiation commutes
with pullback, the restriction of $\omega_{ij}$ to
$Z_{\sm}$ equals $-d\alpha_{ij}=0$.
\end{proof}

Combining the universal bound with the span stratification of
\cref{subsec:span} yields two geometric consequences that we note here.

\begin{corollary}\label{cor:abelian}
Let $\g=\kalg\oplus\p$ be a complex semisimple symmetric pair and put
$D=\dim\p$.
\begin{enumerate}[label=\textup{(\roman*)}]
\item $\Ab_d(\p)=\varnothing$ for every integer $d>D/2$. Equivalently, every
abelian linear subspace contained in $\N(\p)$ has dimension at most $D/2$.
\item If $D$ is even, then $\Ab_{D/2}(\p)$ is a finite set.
\end{enumerate}
\end{corollary}

\begin{proof}
For~(i), let $A\in\Ab_d(\p)$. Then $A^2\subseteq C_2(\N(\p))$, and hence \cref{prop:r2-nilpotent} gives
\[
  2d = \dim A^2 \le \dim C_2(\N(\p)) = D,
\]
so $d\le D/2$.

For~(ii), suppose $D=2d_0$. If $\Ab_{d_0}(\p)$ is empty, there is nothing to
prove. Otherwise, choose an integer $r\ge\max\{2,d_0\}$. By
\cref{prop:span-strata}, the rank-$d_0$ stratum has dimension
\[
  rd_0 + \dim\Ab_{d_0}(\p).
\]
On the other hand, \cref{thm:universal} bounds the dimension of the ambient
variety by
\[
  \dim\Cr(\N(\p)) \le \Bigl\lfloor\frac{rD}{2}\Bigr\rfloor = rd_0.
\]
Since the stratum is contained in $\Cr(\N(\p))$, we obtain
$rd_0 + \dim\Ab_{d_0}(\p) \le rd_0$, which forces $\dim\Ab_{d_0}(\p) = 0$.
Since $\Ab_{d_0}(\p)$ is a projective algebraic variety, it 
consists of only finitely many closed points.
\end{proof}

\begin{corollary}\label{cor:Ab-bound}
For every integer $d\ge2$,
\[
  \dim\Ab_d(\p)
  \le
  \Bigl\lfloor\frac{d\dim\p}{2}\Bigr\rfloor - d^2,
\]
whenever $\Ab_d(\p)$ is nonempty. Moreover,
\[
  \Ab_1(\p) \simeq \mathbb P\N(\p),
\]
and hence, whenever this variety is nonempty,
\[
  \dim\Ab_1(\p)
  = \dim\p - \rk(\g,\kalg) - 1
  \le \dim\p - 2.
\]
\end{corollary}

\begin{proof}
For $d\ge2$, we apply \cref{prop:span-strata} with $r=d$. The rank-$d$
stratum in $C_d(\N(\p))$ has dimension
\[
  d^2 + \dim\Ab_d(\p).
\]
Since $d\ge2$, \cref{thm:universal} applies to $C_d(\N(\p))$ and gives
\[
  d^2 + \dim\Ab_d(\p) \le \dim C_d(\N(\p)) \le \Bigl\lfloor\frac{d\dim\p}{2}\Bigr\rfloor.
\]
Subtracting $d^2$ yields the asserted inequality.

The assertion for $d=1$ follows directly from the definition: every
$1$-dimensional subspace is abelian, so $\Ab_1(\p)$ is the variety of lines in
$\p$ contained in the null cone $\N(\p)$, which is canonically isomorphic to
$\mathbb P\N(\p)$. The dimension formula then follows from
\eqref{eq:KR}. Nonemptiness of $\Ab_1(\p)$ implies $\p\ne0$, so $\rk(\g,\kalg)\ge1$,
giving the upper bound $\dim\p - 2$.
\end{proof}

\section{Hermitian nilpotent commuting varieties}\label{sec:hermitian}

A Hermitian three-grading supplies two complementary abelian
nilpotent subspaces of half the dimension of $\p$.
They attain the universal bound. The equality case of the
projection lemma then determines all maximal-dimensional
components for $r\ge3$.

\subsection{Dimension and maximal-dimensional components}

\begin{theorem}\label{thm:hermitian}
Let $\g=\pmns\oplus\kalg\oplus\pp$ be a complex semisimple
Hermitian symmetric pair. For every $r\ge2$,
\[
 \dim\Cr(\N(\p))=r\dim\pp=\frac r2\dim\p.
\]
\end{theorem}

\begin{proof}
The inclusion $(\pp)^r\subseteq\Cr(\N(\p))$ gives the lower
bound $r\dim\pp$. Since $\dim\p=2\dim\pp$, the upper bound in
\cref{thm:universal} has exactly this value.
\end{proof}

For $r=2$ this is the Hermitian case of Bulois's theorem
\cite[Proposition~1.5]{Bulois}, which was used above.
The following theorem separates the existence of the two linear
components for $r\ge2$ from their exhaustiveness for $r\ge3$.

\begin{theorem}\label{thm:hermitian-max-components}
Let $(\g,\kalg)$ be a nonzero complex Hermitian symmetric pair
of simple type, and put $d=\dim\pp$.
\begin{enumerate}[label=\textup{(\roman*)}]
\item For every $r\ge2$, the subvarieties $(\pp)^r$ and
$(\pmns)^r$ are distinct maximal-dimensional irreducible
components of $\Cr(\N(\p))$, each of dimension $rd$.
\item For every $r\ge3$, these are the only
maximal-dimensional irreducible components of $\Cr(\N(\p))$.
\end{enumerate}
\end{theorem}

\begin{proof}
For \textup{(i)}, both subvarieties are closed and irreducible,
and each has dimension $rd$. By \cref{thm:hermitian}, this is
the dimension of $\Cr(\N(\p))$. Each is therefore an
irreducible component: an irreducible closed subvariety
properly contained in an irreducible component has strictly
smaller dimension. The two subvarieties are distinct because
$\pp\cap\pmns=0$ and $d>0$. This proves \textup{(i)},
including the case $r=2$.

For \textup{(ii)}, let $r\ge3$, and let $Z$ be an irreducible
component of dimension $rd$. Each two-coordinate projection
has closure contained in $C_2(\N(\p))$, which has dimension
$2d$. The equality statement in \cref{lem:pair-projection}
therefore gives
\[
 Z=Y_1\times\cdots\times Y_r,
 \qquad Y_i=\overline{\pi_i(Z)},\qquad \dim Y_i=d.
\]

Both $K$ and $\GL_r$ are connected and preserve the ambient
variety by \cref{lem:symmetries}. Hence each group preserves
$Z$. Indeed, if a connected algebraic group $H$ acts on a
variety, the closure of the image of $H\times Z$ is an
irreducible closed subset containing the irreducible component
$Z$; it must consequently equal $Z$.

Permutation matrices show that $Y_1=\cdots=Y_r$; write
$Z=Y^r$. The diagonal matrices in $\GL_r$ show that $Y$ is
stable under multiplication by nonzero scalars. Since $Y$ is
nonempty and closed, it also contains zero. For $x,y\in Y$,
the tuple $(x,y,0,\ldots,0)$ belongs to $Z$. Applying the
matrix that replaces its first coordinate by the sum of its
first two coordinates shows that $x+y\in Y$. Thus $Y=A$ is
a vector subspace of dimension $d$.

The inclusion $A^r\subseteq\Cr(\N(\p))$ implies that every
element of $A$ is nilpotent and that $[A,A]=0$.
Moreover, $K$-stability of $Z$ implies $K$-stability of $A$.
Let $\xi$ be the grading element. Since $A$ is stable under
$\operatorname{ad}\xi$, its decomposition into eigenspaces is
\[
 A=(A\cap\pp)\oplus(A\cap\pmns).
\]
Each intersection is a $K$-submodule. By
\cref{lem:hermitian-basic}, the modules $\pp$ and $\pmns$
are irreducible, so each intersection is either zero or the
whole corresponding summand. Since both summands have
dimension $d$ and $\dim A=d$, precisely one is present.
Therefore $A=\pp$ or $A=\pmns$, proving \textup{(ii)}.
\end{proof}

\begin{remark}\label{rem:hermitian-r2}
Part \textup{(i)} applies to $r=2$, but the uniqueness assertion
in part \textup{(ii)} does not extend to that case in general.
The maximal-dimensional components for two entries already
have the orbit-theoretic description of Bulois
\cite[Proposition~1.5]{Bulois}:
\[
 \overline{K\cdot\bigl(\{e\}\times\mathfrak z_\p(e)\bigr)},
\]
where $e$ runs through representatives of the
$\p$-distinguished nilpotent $K$-orbits.
Here $e$ is called $\p$-distinguished when
$\mathfrak z_\p(e)\subseteq\N(\p)$.
Thus additional maximal-dimensional components may occur at
$r=2$; the following example exhibits one.
\end{remark}

\begin{example}\label{ex:hermitian-r2-extra}
Consider the pair of type $\mathrm{AIII}(1,2)$. Relative to
$\C^3=\C\oplus\C^2$, an element of $\p$ has the form
\[
 x=\begin{pmatrix}0&b\\c&0\end{pmatrix},
 \qquad b\in\operatorname{Mat}_{1\times2}(\C),\quad
 c\in\operatorname{Mat}_{2\times1}(\C).
\]
We have $x^3=(bc)x$, and $x$ is nilpotent if and only if
$bc=0$. Consequently, $\N(\p)$ is the irreducible quadric
cone $b_1c_1+b_2c_2=0$ in $\C^4$, of dimension three.
The closed irreducible subvariety
\[
 Z=\overline{\{(\lambda x,\mu x):
            x\in\N(\p),\ (\lambda,\mu)\in\C^2\}}
 \subseteq C_2(\N(\p))
\]
has dimension four. Indeed, the parametrisation has
one-dimensional fibres over nonzero pairs: two presentations
differ precisely by rescaling $x$ and inversely rescaling
$(\lambda,\mu)$. Since $d=2$, this equals
$\dim C_2(\N(\p))$, so $Z$ is an irreducible component.
For
\[
 x_0=\begin{pmatrix}0&1&0\\0&0&0\\1&0&0\end{pmatrix},
\]
we have $x_0\in\N(\p)$ and both off-diagonal blocks are
nonzero. Hence $(x_0,x_0)\in Z$ belongs to neither
$(\pp)^2$ nor $(\pmns)^2$. This proves that $Z$ is an
additional maximal-dimensional component.
\end{example}

By \cref{lem:products}, the classification extends to semisimple
Hermitian pairs, one simple factor at a time.
\begin{corollary}\label{cor:hermitian-product-components}
Suppose that, after discarding factors with zero $(-1)$-eigenspace,
$(\g,\kalg)$ is a direct sum of $t$ nonzero simple Hermitian pairs.
For $r\ge3$, its maximal-dimensional nilpotent commuting components
are precisely
\[
 A_{\boldsymbol\epsilon}^{\,r},\qquad
 A_{\boldsymbol\epsilon}
 =\bigoplus_{\nu=1}^t\mathfrak p_\nu^{\epsilon_\nu},
 \qquad \boldsymbol\epsilon\in\{+,-\}^t.
\]
There are exactly $2^t$ such components.
The $A_{\boldsymbol\epsilon}$ are also precisely the
half-dimensional abelian nilpotent subspaces of $\p$.
\end{corollary}

\begin{proof}
The component statement follows from \cref{lem:products} and
\cref{thm:hermitian-max-components}. If $A$ is a half-dimensional
abelian nilpotent subspace, then $A^3$ is a closed irreducible
subvariety of maximal dimension in $C_3(\N(\p))$.
It is one of the listed components, and projection onto a
coordinate gives $A=A_{\boldsymbol\epsilon}$.
\end{proof}
Combining \cref{thm:hermitian} with the isotropy statement of
\cref{prop:universal-isotropy} gives a Lagrangian property.
\begin{corollary}\label{cor:hermitian-lagrangian}
Let $(\g,\kalg)$ be a complex semisimple Hermitian symmetric pair
and let $r\ge2$.
Every irreducible component of $\Cr(\N(\p))$ is simultaneously
isotropic for all $B|_\p\otimes J$.
If $r$ is even and $J$ is nondegenerate, every
maximal-dimensional component is Lagrangian for this form.
\end{corollary}

\begin{proof}
The isotropy follows from \cref{prop:universal-isotropy}.
For even $r$ and nondegenerate $J$, $B|_\p\otimes J$ is
symplectic, and \cref{thm:hermitian} gives half its ambient
dimension.
\end{proof}

\subsection{Abelian nilpotent parameter spaces}

The dimension formula also determines the dimensions of the
Grassmannian parameter spaces introduced in \cref{subsec:span}.

\begin{proposition}\label{prop:Ab-hermitian}
Let $(\g,\kalg)$ be a nonzero complex semisimple Hermitian pair,
and put $d_0=\frac12\dim\p$.
After discarding factors with zero $(-1)$-eigenspace, let $t$ be
the number of simple Hermitian factors.
Then
\[
 a(\p)=d_0,\qquad
 \Ab_{d_0}(\p)=
 \{A_{\boldsymbol\epsilon}:
                  \boldsymbol\epsilon\in\{+,-\}^t\}.
\]
Moreover, for $2\le e\le d_0$,
\[
 \dim\Ab_e(\p)=e(d_0-e).
\]
\end{proposition}

\begin{proof}
The half-dimensional bound is \cref{cor:abelian}, and $\pp$
attains it. The description of $\Ab_{d_0}(\p)$ is
\cref{cor:hermitian-product-components}.
For $2\le e\le d_0$, the closed inclusion
$\Gr(e,\pp)\subseteq\Ab_e(\p)$ gives the lower bound
$e(d_0-e)$. The upper bound in \cref{cor:Ab-bound} is
$\lfloor e(2d_0)/2\rfloor-e^2=e(d_0-e)$.
\end{proof}

\begin{remark}\label{rem:Ab-hermitian}
For $2\le e\le d_0$, each closed irreducible subvariety
$\Gr(e,A_{\boldsymbol\epsilon})$ has dimension $e(d_0-e)$,
so it is a maximal-dimensional component of $\Ab_e(\p)$.
These subvarieties are distinct: the span of all the
$e$-planes in $\Gr(e,A_{\boldsymbol\epsilon})$ recovers
$A_{\boldsymbol\epsilon}$.
Consequently, $\Ab_e(\p)$ has at least $2^t$
maximal-dimensional components. This observation alone does
not classify all its maximal-dimensional components.

For $e=1$, however,
\[
 \Ab_1(\p)\cong\mathbb P\N(\p),\qquad
 \dim\Ab_1(\p)=2d_0-\ell-1,\qquad
 \ell=\rk(\g,\kalg).
\]
For a simple Hermitian pair with $d_0\ge2$ one has $\ell<d_0$, by
\eqref{eq:d-ab} in \cref{subsec:codim-estimate} below; so this dimension
exceeds $\dim\Gr(1,\pp)=d_0-1$. Thus the restriction $e\ge2$ in the preceding formula is essential.
\end{remark}

\section{The full higher commuting variety}\label{sec:full}

We now drop the nilpotency condition and study the full higher
commuting variety. Recall from \cref{def:Cr} that
\[
  \Cr(\p)
  =
  \bigl\{
    (x_1,\dots,x_r)\in\p^r:
    [x_i,x_j]=0
    \text{ for all }1\le i<j\le r
  \bigr\},
\]
regarded throughout as a reduced closed subvariety of \(\p^r\).

We cover \(\Cr(\p)\) by finitely many closed families indexed by
centralizers of semisimple elements. This reduces the dimension
problem to nilpotent higher commuting varieties for smaller symmetric
pairs. The first three subsections apply to arbitrary complex
semisimple symmetric pairs. Beginning with
\cref{subsec:codim-estimate}, we impose the Hermitian hypothesis and
obtain an exact dimension formula.

Throughout, \(\g=\kalg\oplus\p\) is a nonzero complex semisimple symmetric pair,
\(G\) is the adjoint group of \(\g\), and \(K=(G^{\theta})^{\circ}\), so that
\(\operatorname{Lie}K=\kalg\). We fix a Cartan subspace \(\mathfrak a\subseteq\p\)
and write \(\ell=\rk(\g,\kalg)=\dim\mathfrak a\).

\subsection{Restricted roots}\label{subsec:restricted}

We recall the standard apparatus \cite{KostantRallis},
\cite[Chapter~VII, \S2]{Helgason}. Since \(\mathfrak a\) is abelian and consists of
semisimple elements, \(\operatorname{ad}\mathfrak a\) acts diagonalisably on \(\g\), giving the
\emph{restricted root decomposition}
\begin{equation}\label{eq:restricted-decomp}
\g = \mathfrak{z}_{\g}(\mathfrak a)\oplus\bigoplus_{\alpha\in\Sigma}\g_{\alpha},
\qquad
\g_{\alpha}=\{x\in\g : [h,x]=\alpha(h)x\ \forall h\in\mathfrak a\},
\end{equation}
where \(\Sigma\subseteq\mathfrak a^{*}\setminus\{0\}\) is the restricted root
system. We use the following standard facts.
\begin{enumerate}[label=\textup{(R\arabic*)},leftmargin=*]
\item\label{R1} \(\mathfrak{z}_{\p}(\mathfrak a)=\mathfrak a\).
\item\label{R2} \(\theta(\g_{\alpha})=\g_{-\alpha}\), and therefore
\((\g_{\alpha}\oplus\g_{-\alpha})\cap\p=\{x-\theta x:x\in\g_{\alpha}\}\) has
dimension \(\dim\g_{\alpha}\). Consequently,
\begin{equation}\label{eq:dimp-roots}
\dim\p = \ell+\sum_{\alpha\in\Sigma^{+}}\dim\g_{\alpha}
\end{equation}
for any choice of positive system \(\Sigma^{+}\).
\item\label{R3} Every semisimple element of \(\p\) is \(K\)-conjugate to an element
of \(\mathfrak a\), and \(N_K(\mathfrak a)/Z_K(\mathfrak a)\) is the finite little Weyl group
\(W=W(\Sigma)\). For \(h\in\mathfrak a\),
\[
  \mathfrak{z}_\g(h) = \mathfrak{z}_\g(\mathfrak a)
  \oplus\bigoplus_{\alpha(h)=0}\g_\alpha.
\]
\end{enumerate}

Let \(\Delta=\{\alpha_1,\dots,\alpha_\ell\}\) be the simple roots of
\(\Sigma^{+}\). For \(I\subseteq\Delta\), we write
\(\Sigma_I=\Sigma\cap\Z I\) for the standard parabolic subsystem generated
by \(I\), put $\Sigma_I^+=\Sigma_I\cap\Sigma^+$, and set
\(\mathfrak a^{I}=\{u\in\mathfrak a : \alpha(u)=0\ \forall\alpha\in I\}\).

\subsection[Jordan decomposition and p-Levi subalgebras]
{Jordan decomposition and \texorpdfstring{$\p$}{p}-Levi subalgebras}
\label{subsec:plevi} The semisimple parts of a commuting tuple in $\p$ commute and determine
a centraliser in $\g$. We introduce the class of centralisers that
arises in this way.

\begin{definition}\label{def:pLevi} 
A \emph{\(\p\)-Levi subalgebra} of \((\g,\kalg)\) is a subalgebra of the form
\(\mathfrak l=\mathfrak{z}_{\g}(h)\) with \(h\in\p\) semisimple. Such an \(\mathfrak l\) is
reductive and \(\theta\)-stable (since \(\theta(\mathfrak l) = \mathfrak{z}_\g(\theta h) = \mathfrak{z}_\g(-h) = \mathfrak l\)), so
\[
\mathfrak l = \mathfrak{z}(\mathfrak l)\oplus[\mathfrak l,\mathfrak l],
\qquad
\mathfrak l\cap\p = \mathfrak c_{\mathfrak l}\oplus\p_{\mathfrak l},
\]
where
\[
\mathfrak c_{\mathfrak l} = \mathfrak{z}(\mathfrak l)\cap\p,
\qquad
\p_{\mathfrak l} = [\mathfrak l,\mathfrak l]\cap\p,
\qquad
c_{\mathfrak l} = \dim\mathfrak c_{\mathfrak l},
\qquad
\delta_{\mathfrak l} = \dim\p_{\mathfrak l}.
\]
We use the notation
\[
  \N(\p_{\mathfrak l})
  :=
  \N(\mathfrak l_{\mathrm{ss}})
  \cap\p_{\mathfrak l},
  \qquad
  \mathfrak l_{\mathrm{ss}}
  :=
  [\mathfrak l,\mathfrak l].
\]
Thus \(\N(\p_{\mathfrak l})\) is the nilpotent cone of the
\((-1)\)-eigenspace of the induced semisimple symmetric pair.
In particular, if
\(\kalg_{\mathfrak l}=\mathfrak l_{\mathrm{ss}}\cap\kalg\), then
\((\mathfrak l_{\mathrm{ss}},\kalg_{\mathfrak l})\) is again a complex
semisimple symmetric pair (with the convention that if \(\mathfrak l_{\mathrm{ss}}=0\), then \(\p_{\mathfrak l}=0\) and \(\dim\Cr(\N(\p_{\mathfrak l}))=0\)). Its \((-1)\)-eigenspace is
\(\p_{\mathfrak l}\), so \cref{thm:universal} applies to it.
\end{definition}

\begin{lemma}\label{lem:jordan-plevi}
Let \(\mathbf x=(x_1,\dots,x_r)\in\Cr(\p)\) and let \(x_i=s_i+n_i\) be the Jordan
decomposition in \(\g\). Then:
\begin{enumerate}[label=\textup{(\roman*)}]
\item \(s_i,n_i\in\p\) for every \(i\);
\item the \(2r\) elements \(s_1,\dots,s_r,n_1,\dots,n_r\) commute pairwise;
\item \(\mathfrak l:=\mathfrak{z}_{\g}(s_1,\dots,s_r)\) is a \(\p\)-Levi subalgebra;
\item \(s_i\in\mathfrak c_{\mathfrak l}\) and \(n_i\in\p_{\mathfrak l}\) for every
\(i\); in particular,
\((n_1,\dots,n_r)\in\Cr\bigl(\N(\p_{\mathfrak l})\bigr)\).
\end{enumerate}
\end{lemma}

\begin{proof}
(i) We have \(-x_i=\theta(x_i)=\theta(s_i)+\theta(n_i)\), where \(\theta(s_i)\) is
semisimple and \(\theta(n_i)\) is nilpotent and commutes with \(\theta(s_i)\).
By uniqueness of the Jordan decomposition of \(-x_i = (-s_i) + (-n_i)\),
we obtain \(\theta(s_i)=-s_i\) and \(\theta(n_i)=-n_i\), so \(s_i, n_i \in \p\).

(ii) In the faithful adjoint representation,
\[
  \operatorname{ad}(x_i) = \operatorname{ad}(s_i) + \operatorname{ad}(n_i)
\]
is the Jordan decomposition of \(\operatorname{ad}(x_i)\). In characteristic zero,
the semisimple and nilpotent parts of an endomorphism are polynomials without
constant term in that endomorphism. Consequently, the semisimple and nilpotent
parts of commuting endomorphisms commute pairwise. Since \(\g\) is
semisimple, \(\operatorname{ad}\colon\g\to\End(\g)\) is injective, and the assertion
follows.

(iii) The span \(\mathfrak a'=\langle s_1,\dots,s_r\rangle\) is an abelian
subspace of \(\p\) consisting of commuting semisimple elements, so every element
of \(\mathfrak a'\) is semisimple. By the conjugacy theorem for Cartan
subspaces \cite{KostantRallis}, \(\mathfrak a'\) is contained in a Cartan
subspace of \(\p\). After \(K\)-conjugation, we may assume that
\(\mathfrak a'\subseteq\mathfrak a\). By \eqref{eq:restricted-decomp},
\[
\mathfrak{z}_{\g}(\mathfrak a') = \mathfrak{z}_{\g}(\mathfrak a)\oplus
\bigoplus_{\alpha|_{\mathfrak a'}=0}\g_{\alpha},
\qquad
\mathfrak{z}_{\g}(h) = \mathfrak{z}_{\g}(\mathfrak a)\oplus
\bigoplus_{\alpha(h)=0}\g_{\alpha}\qquad(h\in\mathfrak a').
\]
Because \(\Sigma\) is finite, a generic element \(h=\sum_i c_i s_i\in\mathfrak a'\)
satisfies \(\alpha(h)\ne0\) for every \(\alpha\in\Sigma\) that does not vanish
identically on \(\mathfrak a'\). For such an \(h\), the two right-hand sides
coincide, showing that \(\mathfrak l=\mathfrak{z}_{\g}(\mathfrak a')=\mathfrak{z}_{\g}(h)\)
is a \(\p\)-Levi subalgebra.

(iv) By definition of \(\mathfrak l = \mathfrak{z}_\g(s_1,\dots,s_r)\), every
element of \(\mathfrak l\) commutes with each \(s_i\). Since the \(s_j\) commute
pairwise, each \(s_i \in \mathfrak l\), which implies
\(s_i \in \mathfrak{z}(\mathfrak l) \cap \p = \mathfrak c_{\mathfrak l}\).

Each \(n_i\) lies in \(\mathfrak l\cap\p\) by (ii). Write \(n_i=z+y\) with
\(z\in\mathfrak c_{\mathfrak l}\) and \(y\in\p_{\mathfrak l}\), and let \(y=y_s+y_n\)
be the Jordan decomposition of \(y\) inside the semisimple algebra
\([\mathfrak l,\mathfrak l]\), which coincides with its Jordan decomposition in \(\g\).
Because \(\mathfrak l = \mathfrak{z}_\g(h)\) with \(h\) semisimple in semisimple \(\g\),
its center \(\mathfrak{z}(\mathfrak l)\) is a toral subalgebra, so \(z\) is semisimple.
Since \(z \in \mathfrak{z}(\mathfrak l)\), it commutes with \(y_s\) and \(y_n\).
Thus, \(z+y_s\) is semisimple, \(y_n\) is nilpotent, the two commute, and
\(n_i=(z+y_s)+y_n\). Uniqueness of the Jordan decomposition of the nilpotent
element \(n_i\) forces \(z+y_s=0\). By directness of
\(\mathfrak{z}(\mathfrak l)\oplus[\mathfrak l,\mathfrak l]\), this yields \(z=0\)
and \(y_s=0\). Hence \(n_i=y_n\in\p_{\mathfrak l}\). Finally, since the \(n_i\)
commute and are nilpotent in \([\mathfrak l,\mathfrak l]\), we have
\((n_1,\dots,n_r)\in\Cr(\N(\p_{\mathfrak l}))\).
\end{proof}

Up to $K$-conjugacy, $\p$-Levi subalgebras are described by subsets of
the simple restricted roots, and their invariants $c_{\mathfrak l}$ and
$\delta_{\mathfrak l}$ can be read off from restricted-root data.

\begin{lemma}\label{lem:plevi-roots}
Every \(\p\)-Levi subalgebra of \((\g,\kalg)\) is \(K\)-conjugate to a
subalgebra of the form
\[
\mathfrak l_I = \mathfrak{z}_{\g}(\mathfrak a)\oplus\bigoplus_{\alpha\in\Sigma_I}
\g_{\alpha},\qquad I\subseteq\Delta.
\]
Conversely, every such \(\mathfrak l_I\) is a \(\p\)-Levi subalgebra. In
particular, there are only finitely many \(K\)-conjugacy classes of
\(\p\)-Levi subalgebras. Moreover, writing \(s=|I|\),
\begin{equation}\label{eq:c-delta}
\mathfrak c_{\mathfrak l_I} = \mathfrak a^{I},
\qquad
c_{\mathfrak l_I} = \ell-s,
\qquad
\delta_{\mathfrak l_I}
= s + \sum_{\alpha\in\Sigma_I^{+}}\dim\g_{\alpha}.
\end{equation}
\end{lemma}

\begin{proof}
The classification is the restricted-root form of the description
of semisimple centralizers by Satake subdiagrams; see
\cite[Proposition~1.5]{PanyushevYakimova}.
For completeness, after $K$-conjugation let the defining semisimple
element lie in $\mathfrak a$. Choose a real form
$\mathfrak a_{\R}\subseteq\mathfrak a$ on which all restricted
roots are real-valued, and write $h=u+iv$ with
$u,v\in\mathfrak a_{\R}$.
Then $\alpha(h)=0$ if and only if
$\alpha(u)=\alpha(v)=0$.
For a generic real number $t$, the element $h'=u+tv$
has exactly the same vanishing restricted roots as $h$,
and hence the same centralizer in $\g$.
We may now move $h'$ into a closed real Weyl chamber.
Its vanishing roots form $\Sigma_I$ for a subset
$I\subseteq\Delta$, proving the assertion.
Conversely, a generic element of $\mathfrak a^I$ has centralizer
$\mathfrak l_I$.

It remains to establish the formulas in \eqref{eq:c-delta}. If
\(u\in\mathfrak a^{I}\), then \(u\) centralizes
\(\mathfrak{z}_\g(\mathfrak a)\) and acts by zero on each root space
\(\g_\beta\) for \(\beta\in\Sigma_I\subseteq\Z I\) (since \(\beta(u)=0\));
thus \(u\in\mathfrak{z}(\mathfrak l_I)\cap\p\). Conversely, if
\(u\in\mathfrak{z}(\mathfrak l_I)\cap\p\), then \(u\) commutes with
\(\mathfrak a\) (since \(\mathfrak a\subseteq\mathfrak{z}_\g(\mathfrak a)\subseteq\mathfrak l_I\)),
so \(u\in\mathfrak{z}_{\p}(\mathfrak a)=\mathfrak a\) by \ref{R1}.
Furthermore, \(u\) centralizes \(\g_{\alpha}\) for every \(\alpha\in I\),
so \(\alpha(u)=0\). Since the simple roots in \(I\) are linearly
independent, \(\mathfrak c_{\mathfrak l_I} = \mathfrak a^I\) and
\(c_{\mathfrak l_I} = \dim\mathfrak a^{I} = \ell-s\).

Finally, by \ref{R2}, the subspace
\[
\mathfrak l_I\cap\p = \mathfrak a \oplus \bigoplus_{\alpha\in\Sigma_I^{+}}
(\g_{\alpha}\oplus\g_{-\alpha})\cap\p
\]
has dimension \(\ell + \sum_{\alpha\in\Sigma_I^{+}}\dim\g_{\alpha}\).
Because \(\mathfrak l_I\cap\p = \mathfrak c_{\mathfrak l_I}\oplus\p_{\mathfrak l_I}\),
we obtain
\[
\delta_{\mathfrak l_I}
= \dim(\mathfrak l_I\cap\p) - c_{\mathfrak l_I}
= \Bigl(\ell + \sum_{\alpha\in\Sigma_I^{+}}\dim\g_{\alpha}\Bigr) - (\ell - s)
= s + \sum_{\alpha\in\Sigma_I^{+}}\dim\g_{\alpha},
\]
which completes the proof.
\end{proof}

We also need the dimension of the space of $K$-conjugates of a
$\p$-Levi subalgebra.

\begin{lemma}\label{lem:plevi-orbit-dim}
For every \(\p\)-Levi subalgebra \(\mathfrak l\),
\[
\dim K - \dim N_K(\mathfrak l) = \dim\p - c_{\mathfrak l} - \delta_{\mathfrak l}.
\]
\end{lemma}

\begin{proof}
Write \(\mathfrak l = \mathfrak{z}_\g(h)\) with \(h\in\p\) semisimple. First,
the Lie algebra of \(N_K(\mathfrak l)\) is
\(\kalg\cap\mathfrak n_{\g}(\mathfrak l) = \kalg\cap\mathfrak l\), since
\(\mathfrak l\) is self-normalizing in \(\g\): if \(y\in\mathfrak n_\g(\mathfrak l)\),
then \([y,h]\in\mathfrak l = \mathfrak{z}_\g(h)\), so \((\operatorname{ad} h)^2(y) = 0\);
since \(\operatorname{ad} h\) is semisimple, this forces \([y,h]=0\), whence
\(y\in\mathfrak l\).

Next, \(\operatorname{ad} h\) maps \(\kalg\to\p\) and \(\p\to\kalg\). These two
linear maps are mutual adjoints up to sign with respect to the nondegenerate
symmetric forms \(B|_{\kalg}\) and \(B|_{\p}\), since invariance gives
\[
B([h,x],y) = -B(x,[h,y]) \qquad (x\in\kalg,\, y\in\p).
\]
Consequently, they have equal rank:
\(\operatorname{rk}(\operatorname{ad} h|_{\kalg}) = \operatorname{rk}(\operatorname{ad} h|_{\p})\).
Their kernels are
\(\ker(\operatorname{ad} h|_\kalg) = \kalg\cap\mathfrak l\) and
\(\ker(\operatorname{ad} h|_\p) = \p\cap\mathfrak l\). Therefore,
\[
\dim\kalg - \dim(\kalg\cap\mathfrak l)
= \dim\p - \dim(\p\cap\mathfrak l)
= \dim\p - c_{\mathfrak l} - \delta_{\mathfrak l}.
\]
Since \(\dim K = \dim\kalg\) and \(\dim N_K(\mathfrak l) = \dim(\kalg\cap\mathfrak l)\),
the left-hand side equals \(\dim K - \dim N_K(\mathfrak l)\).
\end{proof}

\subsection[The p-Levi family bound and the case r=2]
{The \texorpdfstring{$\p$}{p}-Levi family bound, and the case
\texorpdfstring{$r=2$}{r=2}}
\label{subsec:plevi-bound}

Choose one representative \(\mathfrak l\) from each of the finitely many
\(K\)-conjugacy classes of \(\p\)-Levi subalgebras. Consider the morphism
\[
\Phi_{\mathfrak l}\colon
K\times\Bigl(\mathfrak c_{\mathfrak l}^{\,r}\times
\Cr\bigl(\N(\p_{\mathfrak l})\bigr)\Bigr)\longrightarrow\p^{r},
\qquad
(k,\mathbf s,\mathbf n)\longmapsto
k\cdot(s_1+n_1,\dots,s_r+n_r),
\]
and let \(X_{\mathfrak l}=\overline{\operatorname{im}\Phi_{\mathfrak l}}\).
The image lies in \(\Cr(\p)\) because \(\mathfrak c_{\mathfrak l}\) centralizes
\(\mathfrak c_{\mathfrak l}\oplus\p_{\mathfrak l}\). Conversely,
\cref{lem:jordan-plevi,lem:plevi-roots} show that every point of \(\Cr(\p)\) lies
in one of these images. Since \(\Cr(\p)\) is closed and only finitely many
\(\p\)-Levi classes occur,
\begin{equation}\label{eq:plevi-cover}
  \Cr(\p)=\bigcup_{\mathfrak l}X_{\mathfrak l},
  \qquad
  \dim\Cr(\p)=\max_{\mathfrak l}\dim X_{\mathfrak l}.
\end{equation}

\begin{proposition}
\label{prop:plevi-family-bound}
For every complex semisimple symmetric pair, every \(r\ge2\), and every
\(\p\)-Levi subalgebra \(\mathfrak l\),
\begin{equation}\label{eq:plevi-family-bound}
  \dim X_{\mathfrak l}
  \le
  \dim\p
  +(r-1)c_{\mathfrak l}
  +
  \left\lfloor
    \frac{(r-2)\delta_{\mathfrak l}}{2}
  \right\rfloor .
\end{equation}
In particular,
\[
  \dim X_{\mathfrak l}
  \le
  \dim\p
  +(r-1)c_{\mathfrak l}
  +\frac{r-2}{2}\delta_{\mathfrak l}.
\]
\end{proposition}

\begin{proof}
We abbreviate \(c=c_{\mathfrak l}\), \(\delta=\delta_{\mathfrak l}\), and put
\(N=N_K(\mathfrak l)^{\circ}\). The connected group \(N\) preserves
\(\mathfrak c_{\mathfrak l}\) and \(\p_{\mathfrak l}\), and acts trivially on
the finite set of irreducible components of \(\Cr(\N(\p_{\mathfrak l}))\);
hence \(N\) preserves every individual component.

Let \(Y\) be an irreducible component of \(\Cr(\N(\p_{\mathfrak l}))\). The group
\(N\) acts freely on \(K\times\mathfrak c_{\mathfrak l}^{\,r}\times Y\) by
\[
  u\cdot(k,\mathbf s,\mathbf n)
  =
  (ku^{-1},u\mathbf s,u\mathbf n).
\]
This action preserves the fibers of \(\Phi_{\mathfrak l}\). Therefore,
\[
  \dim\operatorname{im}
  \bigl(
    \Phi_{\mathfrak l}
    |_{K\times\mathfrak c_{\mathfrak l}^{\,r}\times Y}
  \bigr)
  \le
  \dim K-\dim N+rc+\dim Y.
\]
Since \(\dim N=\dim N_K(\mathfrak l)\), taking the maximum over all irreducible
components \(Y\) yields
\[
  \dim X_{\mathfrak l}
  \le
  \dim K-\dim N_K(\mathfrak l)
  +rc+\dim\Cr\bigl(\N(\p_{\mathfrak l})\bigr).
\]
By \cref{def:pLevi}, the pair \(([\mathfrak l,\mathfrak l],[\mathfrak l,\mathfrak l]\cap\kalg)\)
is a complex semisimple symmetric pair with \((-1)\)-eigenspace \(\p_{\mathfrak l}\).
Applying \cref{thm:universal} to this pair gives
\[
  \dim\Cr\bigl(\N(\p_{\mathfrak l})\bigr)
  \le
  \left\lfloor\frac{r\delta}{2}\right\rfloor .
\]
Combining this with \cref{lem:plevi-orbit-dim}, we obtain
\begin{align*}
  \dim X_{\mathfrak l}
  &\le
  (\dim\p-c-\delta)+rc
  +\left\lfloor\frac{r\delta}{2}\right\rfloor\\
  &=
  \dim\p+(r-1)c
  +\left\lfloor\frac{(r-2)\delta}{2}\right\rfloor ,
\end{align*}
where the last equality follows from \(-\delta + \lfloor r\delta/2 \rfloor = \lfloor (r-2)\delta/2 \rfloor\).
\end{proof}

The Cartan subspace supplies a lower bound, which matches the upper bound when $r=2$.

\begin{lemma}\label{lem:cartan-component}
For every complex semisimple symmetric pair and every $r\ge1$,
the subvariety $\overline{K\cdot\mathfrak a^r}$ is an
irreducible component of $\Cr(\p)$ of dimension
$\dim\p+(r-1)\ell$.
\end{lemma}

\begin{proof}
Let
\[
 \mathfrak a_{\mathrm{reg}}
 =\{h\in\mathfrak a:\alpha(h)\ne0 \text{ for every }\alpha\in\Sigma\}.
\]
The restricted-root decomposition gives
$\mathfrak z_\p(h)=\mathfrak a$ for
$h\in\mathfrak a_{\mathrm{reg}}$.
By \cite{KostantRallis}, the set
\[\p_{\mathrm{rss}}:=K\cdot\mathfrak a_{\mathrm{reg}}
\]
is a nonempty Zariski-open subset of $\p$.
Its elements are precisely the semisimple $x\in\p$ satisfying
$\dim\mathfrak z_\p(x)=\ell$; these are the regular semisimple
elements of the isotropy representation.
The open subset
\[
 U=\{(x_1,\ldots,x_r)\in\Cr(\p):x_1\in\p_{\mathrm{rss}}\}
\]
is the image of
\[
 K\times\bigl(\mathfrak a_{\mathrm{reg}}\times
                    \mathfrak a^{r-1}\bigr)
 \longrightarrow\p^r,
 \qquad (k,\mathbf h)\longmapsto k\mathbf h.
\]
Indeed, after conjugating $x_1$ into $\mathfrak a_{\mathrm{reg}}$,
commutativity forces every other entry into its centralizer
$\mathfrak a$ in $\p$.
Thus $U$ is irreducible, and its closure is
$\overline{K\cdot\mathfrak a^r}$. The closure of a nonempty
irreducible open subset is an irreducible component.

We compute the fibres of the displayed morphism.
If $k\mathbf h=k'\mathbf h'$, then $u=k^{-1}k'$ satisfies
$uh'_1=h_1$. Both first entries are regular, so
\[
 u\mathfrak a=u\mathfrak z_\p(h'_1)
              =\mathfrak z_\p(h_1)=\mathfrak a.
\]
Hence $u\in N_K(\mathfrak a)$. Conversely, every element of
this normalizer gives a point of the same fibre by changing
$(k,\mathbf h)$ to $(ku^{-1},u\mathbf h)$.
Every fibre therefore has dimension $\dim N_K(\mathfrak a)$,
which equals $\dim Z_K(\mathfrak a)$ because the little Weyl
group is finite.

For $\mathfrak l=\mathfrak z_\g(\mathfrak a)$, we have
$c_{\mathfrak l}=\ell$ and $\delta_{\mathfrak l}=0$.
Moreover, the Lie algebras of $N_K(\mathfrak l)$ and
$Z_K(\mathfrak a)$ both equal
$\kalg\cap\mathfrak z_\g(\mathfrak a)$.
By \cref{lem:plevi-orbit-dim},
\[
 \dim K-\dim Z_K(\mathfrak a)=\dim\p-\ell.
\]
The fibre-dimension theorem now gives
\[
 \dim\overline{K\cdot\mathfrak a^r}
 =\dim K-\dim Z_K(\mathfrak a)+r\ell
 =\dim\p+(r-1)\ell.
\]
\end{proof}

At \(r=2\), the term \(\lfloor\tfrac{(r-2)\delta_{\mathfrak l}}{2}\rfloor\) in
\eqref{eq:plevi-family-bound} vanishes, and the upper bound matches the
dimension of the Cartan component. Note that the argument still relies on
\cref{thm:universal}, specialized to the form
\[
  \dim C_2(\N(\p_{\mathfrak l})) \le \delta_{\mathfrak l}.
\]
The dimension formula and the fact that \(\overline{K\cdot\mathfrak a^2}\) is an
irreducible component of maximal dimension were established by
Panyushev~\cite[\S3.5 and Proposition~3.7]{Panyushev94}. The uniqueness of this
maximal-dimensional component is also recorded in
\cite[Corollary~2.3]{SabourinYu}; see also \cite{BuloisIrregular}. We include the short deduction below because
the same strategy, applied for general \(r\), leads to our main theorem in the
Hermitian setting (\cref{thm:full-hermitian}).

\begin{corollary}
\label{cor:r2-general}
For every complex semisimple symmetric pair,
\[
  \dim C_2(\p) = \dim\p + \ell = \dim\p + \rk(\g,\kalg),
\]
and \(\overline{K\cdot\mathfrak a^{2}}\) is the unique irreducible component of
\(C_2(\p)\) of maximal dimension.
\end{corollary}

\begin{proof}
For $r=2$, \cref{prop:plevi-family-bound} gives
\[
 \dim X_{\mathfrak l}\le\dim\p+c_{\mathfrak l}
 \le\dim\p+\ell,
\]
for every
\(\p\)-Levi subalgebra \(\mathfrak l\). Together with the finite cover \eqref{eq:plevi-cover} and
\cref{lem:cartan-component}, this proves the dimension formula.
Every maximal-dimensional component lies in some
$X_{\mathfrak l}$ of dimension $\dim\p+\ell$, so
$c_{\mathfrak l}=\ell$.
By \cref{lem:plevi-roots}, this forces
$\mathfrak l$ to be $K$-conjugate to
$\mathfrak z_\g(\mathfrak a)$, for which
$X_{\mathfrak l}=\overline{K\cdot\mathfrak a^2}$.
This proves uniqueness.
\end{proof}

\begin{remark} 
For the group-type symmetric pair \((\g_0\oplus\g_0,\Delta\g_0)\) of
\cref{ex:groupcase}, \cref{cor:r2-general} gives
$\dim C_2(\g_0)=\dim\g_0+\rk\g_0$, in agreement with
Richardson's theorem \cite{Richardson}. Richardson proves the
stronger assertion that $C_2(\g_0)$ is irreducible. Uniqueness
of a maximal-dimensional component alone does not imply irreducibility.

For general symmetric pairs, \(C_2(\p)\) is frequently reducible
\cite{Panyushev04,PanyushevYakimova}, so the uniqueness statement
concerns only the maximal-dimensional component; components of strictly
smaller dimension can and do occur. For instance, for
\((\mathfrak{sl}_{n+m}, \mathfrak{s}(\mathfrak{gl}_n\oplus\mathfrak{gl}_m))\)
with \(n\ne m\), the number of irreducible components can be arbitrarily large
\cite[Prop.~4.4]{PanyushevYakimova}.
\end{remark}

More generally, combining \cref{prop:plevi-family-bound} and \cref{lem:cartan-component}
supplies explicit two-sided bounds on \(\dim\Cr(\p)\) for any symmetric pair,
governed purely by restricted-root data.

\begin{corollary}\label{cor:general-upper}
For every complex semisimple symmetric pair and every \(r\ge2\),
\[
  \max\Bigl\{
    \dim\Cr(\N(\p)),
    \ \dim\p+(r-1)\ell
  \Bigr\}
  \le
  \dim\Cr(\p)
\]
and
\[
  \dim\Cr(\p)
  \le
  \dim\p+\max_{I\subseteq\Delta}
  \left[
    (r-1)(\ell-|I|)
    +
    \left\lfloor
      \frac{(r-2)\delta_{\mathfrak l_I}}{2}
    \right\rfloor
  \right],
\]
where \(\delta_{\mathfrak l_I} = |I| + \sum_{\alpha\in\Sigma_I^{+}}\dim\g_{\alpha}\).
\end{corollary}

\begin{proof}
For the lower bound, \(\Cr(\N(\p))\) is a closed subvariety of \(\Cr(\p)\), so
\(\dim\Cr(\p) \ge \dim\Cr(\N(\p))\). Furthermore, \cref{lem:cartan-component}
shows that \(\overline{K\cdot\mathfrak a^r}\subseteq\Cr(\p)\) has dimension
\(\dim\p + (r-1)\ell\), which gives the second lower bound.

For the upper bound, the finite union in \eqref{eq:plevi-cover} gives
\(\dim\Cr(\p) = \max_{\mathfrak l}\dim X_{\mathfrak l}\). By
\cref{lem:plevi-roots}, every \(\p\)-Levi subalgebra is \(K\)-conjugate to
\(\mathfrak l_I\) for some subset \(I\subseteq\Delta\), with
\(c_{\mathfrak l_I} = \ell - |I|\) and \(\delta_{\mathfrak l_I}\) given by
\eqref{eq:c-delta}. Applying \cref{prop:plevi-family-bound} to each \(\mathfrak l_I\)
yields the asserted maximum.
\end{proof}

\subsection{The Hermitian codimension estimate}\label{subsec:codim-estimate}

From now on, \((\g,\kalg)\) is a simple Hermitian symmetric pair, and we put
\[
  d = \dim\pp = \frac12\dim\p,
  \qquad
  \ell = \rk(\g,\kalg).
\]
After interchanging the two blocks in type \(\mathrm{AIII}\) if necessary, we
assume \(1\le m\le n\).

The restricted root system \(\Sigma\) for all these symmetric pairs is of type \(C_\ell\) or \(BC_\ell\);
see \cite[Chapter~X, \S6]{Helgason}. Choose a coordinate basis
$e_1,\ldots,e_\ell$ of the real span of the restricted roots.
The positive roots may be normalized as
\begin{equation}\label{eq:hermitian-roots}
  e_i\pm e_j \,\, (1\le i<j\le\ell),
  \,\, 
  2e_i \,\, (1\le i\le\ell), \,\, 
  e_i \,\, (1\le i\le\ell,\ \text{present precisely when } b_{\Sigma}>0),
\end{equation}
with multiplicities
\begin{equation}\label{eq:hermitian-multiplicities}
  \dim\g_{e_i\pm e_j} = a_{\Sigma},
  \qquad
  \dim\g_{2e_i} = 1,
  \qquad
  \dim\g_{e_i} = 2b_{\Sigma},
\end{equation}
where $a_\Sigma\in\Z_{\ge1}$ and $b_\Sigma\in\Z_{\ge0}$.
The last equality is used only when $b_\Sigma>0$.
When $\ell=1$, no roots $e_i\pm e_j$ occur; $a_\Sigma$ is then
the formal value in the table below and does not affect any dimension.
Thus \(\Sigma\) is of type
\(C_\ell\) when \(b_{\Sigma}=0\) and of type \(BC_\ell\) when \(b_{\Sigma}>0\).
For \(\mathrm{BDI}_{\mathrm{Herm}}\), we adopt the \(C_2\) normalization of the
\(B_2\) root system. The characteristic multiplicities are summarized in
\eqref{eq:ab-table}:
\begin{equation}\label{eq:ab-table}
\begin{array}{@{}lccc@{}}
\toprule
\text{type} & \ell & a_{\Sigma} & b_{\Sigma} \\
\midrule
\mathrm{AIII}(m,n),\ 1\le m\le n & m & 2 & n-m \\
\mathrm{BDI}_{\mathrm{Herm}}(n),\ n\ge3 & 2 & n-2 & 0 \\
\mathrm{CI}(n) & n & 1 & 0 \\
\mathrm{DIII}(2\ell) & \ell & 4 & 0 \\
\mathrm{DIII}(2\ell+1) & \ell & 4 & 2 \\
\mathrm{EIII} & 2 & 6 & 4 \\
\mathrm{EVII} & 3 & 8 & 0 \\
\bottomrule
\end{array}
\end{equation}
Counting positive restricted roots with their multiplicities in
\eqref{eq:dimp-roots} yields
\begin{equation}\label{eq:d-ab}
  \dim\p
  = 2d
  = 2\ell + a_{\Sigma}\ell(\ell-1) + 2b_{\Sigma}\ell,
  \qquad
  d
  = \ell + \frac{a_{\Sigma}}2\ell(\ell-1) + b_{\Sigma}\ell.
\end{equation}
These formulas reproduce the values of \(\dim\pp\) recorded in
\cref{subsec:hermitian-classification}. Note that \(a_{\Sigma}\ge1\) in every case.

For \(t\ge0\), define the polynomial functions
\begin{equation}\label{eq:AH}
  A(t) = t + \frac{a_{\Sigma}}2t(t+1),
  \qquad
  H(t) = 2t + a_{\Sigma}t(t-1) + 2b_{\Sigma}t.
\end{equation}
Under this notation, \(\dim\p = H(\ell)\).

The next estimate is the key input for the Hermitian case: it shows that
an intermediate $\p$-Levi subalgebra has large codimension compared with
its central part.
\begin{lemma}\label{lem:hermitian-plevi-codim}
Let \(\mathfrak l\) be a \(\p\)-Levi subalgebra of a simple Hermitian
symmetric pair, and write \(c=c_{\mathfrak l}\), \(\delta=\delta_{\mathfrak l}\).
Then
\[
  \delta=0 \iff c=\ell,
  \qquad\text{and}\qquad
  \dim\p-\delta \ge 4c \quad\text{whenever } 0<c<\ell.
\]
Moreover, equality $\dim\p-\delta=4c$ with $0<c<\ell$
can occur only for the pair of type $\mathrm{CI}(2)$,
equivalently $\mathrm{BDI}_{\mathrm{Herm}}(3)$.
\end{lemma}

\begin{proof}
By \cref{lem:plevi-roots}, after \(K\)-conjugation we may assume that
\(\mathfrak l = \mathfrak l_I\) for some subset of simple roots \(I\subseteq\Delta\).
Set \(s=|I|\). By \eqref{eq:c-delta},
\begin{equation}\label{eq:codim-c-delta}
  c = \ell-s,
  \qquad
  \delta = s + \sum_{\alpha\in\Sigma_I^+}\dim\g_\alpha.
\end{equation}
Since every term in the sum is strictly positive, \(\delta=0\) if and only if
\(s=0\), which is equivalent to \(c=\ell\). This establishes the first
equivalence.

We now assume that \(0<c<\ell\), which is equivalent to \(1\le s\le\ell-1\) and
\(c=\ell-s\). The simple restricted roots of \(\Sigma\) are ordered as
\[
  \alpha_1 = e_1-e_2,\ \ldots,\ \alpha_{\ell-1} = e_{\ell-1}-e_\ell,
\]
together with the special end root \(\alpha_\ell\), which is \(2e_\ell\) in
type \(C_\ell\) and \(e_\ell\) in type \(BC_\ell\).

Any subset \(I\subseteq\Delta\) decomposes into connected components of the
restricted Dynkin diagram. Each connected component of \(I\) that does not
contain the special node \(\alpha_\ell\) is a path of type \(A_t\) for some
\(t\ge1\). Its contribution to the right-hand side of \eqref{eq:codim-c-delta}
is precisely
\[
  t + a_\Sigma\binom{t+1}{2} = t + \frac{a_\Sigma}{2}t(t+1) = A(t).
\]
If a connected component of \(I\) contains the special node \(\alpha_\ell\) and
has rank \(u\ge1\), it is of type \(C_u\) or \(BC_u\), and its contribution to
\eqref{eq:codim-c-delta} is
\[
  u + \Bigl(u + a_\Sigma u(u-1) + 2b_\Sigma u\Bigr) = 2u + a_\Sigma u(u-1) + 2b_\Sigma u = H(u).
\]
Thus, if the components not containing \(\alpha_\ell\) have ranks
\(t_1,\ldots,t_v\), and if \(u\) denotes the rank of the component containing
\(\alpha_\ell\) (with \(u=0\) when no such component exists), then
\[
  s = u + t_1 + \cdots + t_v,
\]
and \(\delta\) decomposes as
\begin{equation}\label{eq:codim-block-decomposition}
  \delta =
  \begin{cases}
    \displaystyle\sum_{j=1}^v A(t_j), & \text{if } u=0,\\[3mm]
    \displaystyle H(u) + \sum_{j=1}^v A(t_j), & \text{if } u>0.
  \end{cases}
\end{equation}

The following two elementary identities allow us to merge these blocks:
\begin{align}
  A(x+y) - A(x) - A(y)
  &= a_\Sigma xy \ge 0,
  \label{eq:codim-A-merge}\\
  H(u+t) - H(u) - A(t)
  &= \frac{t}{2}\bigl(a_\Sigma(t+4u-3) + 4b_\Sigma + 2\bigr) \ge 0,
  \label{eq:codim-H-merge}
\end{align}
valid for all \(x,y\ge0\) and \(u,t\ge1\), since \(a_\Sigma\ge1\) and
\(b_\Sigma\ge0\).

If \(u=0\), repeated application of \eqref{eq:codim-A-merge} yields
\(\delta \le A(s)\). If \(u>0\), repeated application of \eqref{eq:codim-H-merge}
yields \(\delta \le H(s)\). Consequently,
\begin{equation}\label{eq:codim-delta-max}
  \delta \le \max\{A(s), H(s)\}.
\end{equation}

It remains to compare \(A(s)\) and \(H(s)\) with \(H(\ell) = \dim\p\). For
every \(t\ge1\),
\begin{equation}\label{eq:codim-H-increment}
  H(t+1) - H(t) = 2(a_\Sigma t + b_\Sigma + 1) \ge 2(1\cdot 1 + 0 + 1) = 4.
\end{equation}
Since \(\ell = s+c\), telescoping across the \(c\) steps gives
\begin{equation}\label{eq:codim-H-bound}
  H(\ell) - H(s) = \sum_{t=s}^{\ell-1}\bigl(H(t+1) - H(t)\bigr) \ge 4c.
\end{equation}

For the term \(A(s)\), a direct calculation gives
\begin{equation}\label{eq:codim-A-first-step}
  H(s+1) - A(s) - 4
  = \frac12\bigl(a_\Sigma s(s+1) + 4b_\Sigma(s+1) + 2s - 4\bigr).
\end{equation}
For \(s=1\), the parenthetical expression evaluates to
\(2a_\Sigma + 8b_\Sigma - 2 \ge 2(1) - 2 = 0\). For \(s\ge2\), we have
\(2s-4\ge0\) and \(a_\Sigma s(s+1) > 0\). Hence \(H(s+1) - A(s) \ge 4\) for all
\(s\ge1\). Combining this with \eqref{eq:codim-H-increment} for the remaining
\(c-1\) increments yields
\begin{equation}\label{eq:codim-A-bound}
  H(\ell) - A(s)
  = \bigl(H(s+1) - A(s)\bigr) + \sum_{t=s+1}^{\ell-1}\bigl(H(t+1) - H(t)\bigr)
  \ge 4 + 4(c-1) = 4c.
\end{equation}

Combining \eqref{eq:codim-delta-max}, \eqref{eq:codim-H-bound}, and
\eqref{eq:codim-A-bound}, we conclude that
\[
  \dim\p - \delta
  = H(\ell) - \delta
  \ge H(\ell) - \max\{A(s), H(s)\}
  \ge 4c.
\]
For the equality assertion, equality in
$H(\ell)-H(s)\ge4c$ forces every increment
$2(a_\Sigma t+b_\Sigma+1)$, for $s\le t\le\ell-1$,
to equal $4$. Thus $s=c=1$, $a_\Sigma=1$, and $b_\Sigma=0$.
Equality in $H(\ell)-A(s)\ge4c$ first requires
\[
 a_\Sigma s(s+1)+4b_\Sigma(s+1)+2s-4=0.
\]
This again forces $s=1$, $a_\Sigma=1$, and $b_\Sigma=0$;
any further increment is strictly larger than $4$, so $c=1$.
Since $\delta\le\max\{A(s),H(s)\}$, these are the only
possible parameters for $\dim\p-\delta=4c$.
The restricted-root table identifies the pair as $\mathrm{CI}(2)$.
This completes the proof.
\end{proof}

\subsection{The exact dimension formula}
\label{subsec:full-dimension}

The Hermitian three-grading \(\g=\pmns\oplus\kalg\oplus\pp\) provides two closed linear
subvarieties of commuting nilpotent tuples, \((\pmns)^r\) and \((\pp)^r\).
A Cartan subspace \(\mathfrak a\subseteq\p\) provides a third distinguished
closed irreducible subvariety, the Cartan orbit closure
\(\overline{K\cdot\mathfrak a^r}\). The dimensions of these families are
\[
  \dim (\pp)^r = \dim (\pmns)^r = rd,
  \qquad
  \dim \overline{K\cdot\mathfrak a^r} = 2d+(r-1)\ell,
\]
where
\[
  d = \dim\pp = \dim\pmns = \frac12\dim\p,
  \qquad
  \ell = \rk(\g,\kalg).
\]
The next theorem proves that no other \(\p\)-Levi family can exceed these
dimensions, determining the exact dimension of \(\Cr(\p)\) for all \(r\ge2\).

\begin{theorem}\label{thm:full-hermitian}
Let \(\g=\pmns\oplus\kalg\oplus\pp\) be a nonzero complex simple Hermitian
symmetric pair, and set \(d=\dim\pp\) and \(\ell=\rk(\g,\kalg)\). Then, for every
\(r\ge2\),
\[
  \dim\Cr(\p)
  =
  \max\bigl\{
    rd,\,
    2d+(r-1)\ell
  \bigr\}.
\]
Equivalently,
\[
  \dim\Cr(\p)
  =
  \max\bigl\{
    \dim (\pp)^r,\,
    \dim \overline{K\cdot\mathfrak a^r}
  \bigr\}.
\]
\end{theorem}

\begin{proof}
For \(r=2\), since \(\dim\p=2d\) and \(\ell\ge1\), we have
\[
  \max\{2d,\, 2d+\ell\} = 2d+\ell = \dim\p+\ell.
\]
The formula therefore follows from \cref{cor:r2-general}. We may therefore assume
\(r\ge3\).

We first establish the upper bound. Let \(\mathfrak l\) be an arbitrary
\(\p\)-Levi subalgebra, and set \(c=c_{\mathfrak l}\) and
\(\delta=\delta_{\mathfrak l}\). Then \(0\le c\le\ell\) and \(\dim\p=2d\).
By \cref{prop:plevi-family-bound}, the family closure \(X_{\mathfrak l}\) satisfies
\begin{equation}\label{eq:main-family-bound}
  \dim X_{\mathfrak l}
  \le
  2d+(r-1)c+\frac{r-2}{2}\delta.
\end{equation}
We analyze three cases based on the value of \(c\):
\begin{enumerate}[label=\textup{(\alph*)},leftmargin=*]
\item \textbf{Case \(c=\ell\).} By \cref{lem:hermitian-plevi-codim},
\(\delta=0\). Substituting into \eqref{eq:main-family-bound} gives
\[
  \dim X_{\mathfrak l} \le 2d+(r-1)\ell.
\]
\item \textbf{Case \(c=0\).} By \cref{lem:plevi-roots}, \(\mathfrak l=\g\).
Consequently, \(\delta=\dim\p=2d\) and \(X_{\mathfrak l}=\Cr(\N(\p))\).
Then \eqref{eq:main-family-bound} yields
\[
  \dim X_{\mathfrak l} \le 2d+\frac{r-2}{2}(2d) = rd.
\]
(In fact, equality holds by \cref{thm:hermitian}.)
\item \textbf{Case \(0<c<\ell\).} By \cref{lem:hermitian-plevi-codim}, the
codimension satisfies
\[
  2d-\delta = \dim\p-\delta \ge 4c.
\]
Rewriting \eqref{eq:main-family-bound} by subtracting \(rd\), we obtain
\begin{align*}
  \dim X_{\mathfrak l} - rd
  &\le
  (2-r)d + (r-1)c + \frac{r-2}{2}\delta \\
  &=
  (r-1)c - \frac{r-2}{2}(2d-\delta) \\
  &\le
  (r-1)c - \frac{r-2}{2}(4c) \\
  &=
  (r-1)c - 2(r-2)c \\
  &=
  (3-r)c.
\end{align*}
Since \(r\ge3\) and \(c>0\), we have \((3-r)c \le 0\). Thus,
\(\dim X_{\mathfrak l} \le rd\) for every intermediate \(\p\)-Levi subalgebra.
\end{enumerate}

Taking the maximum over the finite closed cover
\(\Cr(\p) = \bigcup_{\mathfrak l} X_{\mathfrak l}\) from \eqref{eq:plevi-cover},
we conclude that
\[
  \dim\Cr(\p)
  \le
  \max\bigl\{
    rd,\,
    2d+(r-1)\ell
  \bigr\}.
\]

For the matching lower bound, \((\pp)^r\) is a closed irreducible subvariety
of \(\Cr(\p)\), so \(\dim\Cr(\p)\ge\dim(\pp)^r = rd\). Furthermore,
\cref{lem:cartan-component} shows that \(\overline{K\cdot\mathfrak a^r}\) is a
closed irreducible subvariety of \(\Cr(\p)\) of dimension
\(\dim\p+(r-1)\ell = 2d+(r-1)\ell\). Hence
\[
  \dim\Cr(\p) \ge \max\bigl\{rd,\, 2d+(r-1)\ell\bigr\},
\]
which proves the equality.
\end{proof}
For the classification of components, we need a strict form of the
preceding estimate for the intermediate families.
\begin{lemma}\label{lem:intermediate-levi}
Let $(\g,\kalg)$ be a nonzero simple Hermitian symmetric pair.
If $r\ge2$ and $\mathfrak l$ is a $\p$-Levi subalgebra with
$0<c_{\mathfrak l}<\ell$, then
\[
 \dim X_{\mathfrak l}<\dim\Cr(\p).
\]
\end{lemma}

\begin{proof}
Write $c=c_{\mathfrak l}$ and $\delta=\delta_{\mathfrak l}$.
For $r=2$, the family bound gives
\[
 \dim X_{\mathfrak l}\le2d+c<2d+\ell=\dim C_2(\p).
\]
For $r\ge4$, the proof of \cref{thm:full-hermitian} gives
\[
 \dim X_{\mathfrak l}-rd\le(3-r)c<0,
\]
whereas $\dim\Cr(\p)\ge rd$.

It remains to consider $r=3$. If the pair is not of type
$\mathrm{CI}(2)$, \cref{lem:hermitian-plevi-codim} gives
$2d-\delta>4c$. Hence
\[
 \dim X_{\mathfrak l}-3d
 \le2c-\frac12(2d-\delta)<0.
\]
For $\mathrm{CI}(2)$, we have $d=3$ and $\ell=2$.
The non-strict estimate gives $\dim X_{\mathfrak l}\le9$,
while \cref{thm:full-hermitian} gives $\dim C_3(\p)=10$.
This proves the strict inequality in all cases.
\end{proof}

\subsubsection{The transition threshold}

The dimensions of the Cartan family \(\overline{K\cdot\mathfrak a^r}\) and the
linear nilpotent families \((\pp)^r\) and \((\pmns)^r\) are
compared by a numerical threshold. Define the ratio
\begin{equation}\label{eq:threshold}
  \rho
  :=
  \frac d\ell
  =
  1+\frac{a_\Sigma}{2}(\ell-1)+b_\Sigma,
\end{equation}
and set
\[
  \tau
  :=
  \begin{cases}
    \displaystyle
    2+\frac{1}{\rho-1},
      & \text{if } \rho>1,\\[3mm]
    +\infty,
      & \text{if } \rho=1.
  \end{cases}
\]

\begin{corollary}\label{cor:threshold}
Let \((\g,\kalg)\) be a nonzero complex simple Hermitian symmetric pair.
For every integer \(r\ge2\),
\[
\begin{aligned}
  rd < 2d+(r-1)\ell &\iff r < \tau,\\
  rd = 2d+(r-1)\ell &\iff r = \tau,\\
  rd > 2d+(r-1)\ell &\iff r > \tau.
\end{aligned}
\]
Consequently,
\[
  \dim\Cr(\p)
  =
  \begin{cases}
    \dim\overline{K\cdot\mathfrak a^r} = 2d+(r-1)\ell, & \text{if } r<\tau,\\[2mm]
    \dim(\pp)^r = rd, & \text{if } r>\tau,
  \end{cases}
\]
and the two dimensions coincide precisely when \(\tau\) is an integer and
\(r=\tau\).

Up to isomorphism, the simple Hermitian symmetric pairs partition into the
following three cases:
\begin{enumerate}[label=\textup{(\roman*)}]
\item \(\rho=1\) occurs uniquely for \((\mathfrak{sl}_2,\mathfrak{so}_2)\); in
this case \(\tau=+\infty\), and the Cartan family has strictly larger dimension
than the linear nilpotent families for every \(r\ge2\).

\item \(\rho=\tfrac32\) occurs uniquely for
\((\mathfrak{sp}_4,\mathfrak{gl}_2) \cong (\mathfrak{so}_5, \mathfrak{so}_3\oplus\mathfrak{so}_2)\);
in this case \(\tau=4\), so the Cartan family strictly dominates for \(r\in\{2,3\}\),
an exact dimensional tie occurs at \(r=4\), and the nilpotent families strictly
dominate for all \(r\ge5\).

\item Every other simple Hermitian symmetric pair satisfies
\[
  \rho\ge2,
  \qquad
  \tau\le3.
\]
In particular, the linear nilpotent families strictly dominate for all \(r\ge4\),
and they dominate or tie with the Cartan family at \(r=3\).
\end{enumerate}
\end{corollary}

\begin{proof}
The difference between the two dimensions is
\[
 rd-\bigl(2d+(r-1)\ell\bigr)
 =\ell\bigl((r-2)(\rho-1)-1\bigr).
\]
For $\rho>1$, this has the sign of $r-\tau$.
For $\rho=1$, it equals $-\ell<0$.
This proves the comparisons.

For the classification, use
$\rho-1=\tfrac12 a_\Sigma(\ell-1)+b_\Sigma$
and \eqref{eq:ab-table}.
Since $a_\Sigma\ge1$ and $b_\Sigma$ is a nonnegative integer,
$\rho=1$ forces $\ell=1$ and $b_\Sigma=0$, giving
$(\mathfrak{sl}_2,\mathfrak{so}_2)$.
If $0<\rho-1<1$, then $\ell\ne1$ because the rank-one entries
have integral $b_\Sigma$; hence
$(\ell,a_\Sigma,b_\Sigma)=(2,1,0)$.
The table identifies this case with $\mathrm{CI}(2)$, for which
$\tau=4$. Every other entry has $\rho\ge2$ and $\tau\le3$.
\end{proof}

\subsubsection{Maximal-dimensional components} The threshold also determines which of the three distinguished families
are components of maximal dimension.

\begin{corollary}[Maximal-dimensional full components]
\label{cor:full-components}
Under the hypotheses of \cref{thm:full-hermitian}, put
\[
 L_r=rd,\qquad S_r=2d+(r-1)\ell,\qquad
 \mathcal T_r=\overline{K\cdot\mathfrak a^r},
\]
where $\mathfrak a$ is a Cartan subspace of $\p$.
The maximal-dimensional irreducible components of $\Cr(\p)$
are precisely:
\begin{enumerate}[label=\textup{(\roman*)}]
\item $(\pp)^r$ and $(\pmns)^r$ if $L_r>S_r$;
\item $\mathcal T_r$ if $L_r<S_r$;
\item $(\pp)^r$, $(\pmns)^r$, and $\mathcal T_r$ if $L_r=S_r$.
\end{enumerate}
Equality occurs exactly in the following cases, up to isomorphism:
\[
\begin{array}{ll}
r=3:&
\mathrm{AIII}(1,2),\quad
\mathrm{AIII}(2,2)\cong\mathrm{BDI}_{\mathrm{Herm}}(4),\quad
\mathrm{CI}(3),\\[1mm]
r=4:&\mathrm{CI}(2).
\end{array}
\]
\end{corollary}

\begin{proof}
The two linear varieties have dimension $L_r$, and
$\mathcal T_r$ has dimension $S_r$ by
\cref{lem:cartan-component}. They are pairwise distinct:
$\pp\cap\pmns=0$, while $\mathcal T_r$ contains tuples
with a nonzero semisimple entry.

For $r=2$, the unique maximal-dimensional component is
$\mathcal T_2$ by \cref{cor:r2-general}.
Let $r\ge3$. By \cref{lem:intermediate-levi}, the intermediate
$\p$-Levi families have dimension strictly less than
$\dim\Cr(\p)$. The family with $c=\ell$ is $\mathcal T_r$,
and the family with $c=0$ is $\Cr(\N(\p))$.
Any maximal-dimensional irreducible component of $\Cr(\p)$
is contained in one member of the finite closed cover
\eqref{eq:plevi-cover}. The intermediate families have just
been excluded. If it is contained in the nilpotent family,
it must have dimension $L_r$ and hence must be one of the two
linear components, by \cref{thm:hermitian-max-components}.
If it is contained in $\mathcal T_r$, irreducibility forces
it to equal $\mathcal T_r$.
Conversely, each listed subvariety whose dimension equals
$\max\{L_r,S_r\}$ is an irreducible component.
This proves the three alternatives.

Finally, equality is impossible for $r=2$. For $r\ge3$,
$L_r=S_r$ is equivalent to
$d/\ell-1=1/(r-2)$. By \eqref{eq:d-ab}, $d/\ell-1=\tfrac{a_\Sigma}{2}(\ell-1)+b_\Sigma$ lies in
$\tfrac12\Z_{\ge0}$, so $1/(r-2)\in\{1,\tfrac12\}$ and $r\in\{3,4\}$. Substitution from \eqref{eq:ab-table}
gives the displayed list.
\end{proof}

\subsubsection{Products and explicit formulas} Finally, we extend the formula to semisimple Hermitian pairs and make it
explicit for each simple type.

\begin{corollary}\label{cor:full-products}
Let $r\ge2$, and let
\[
  (\g,\kalg) = \bigoplus_{\nu=1}^t (\g_\nu,\kalg_\nu)
\]
be a direct sum of simple Hermitian symmetric pairs. Put
\[
  d_\nu = \dim\pp_\nu,
  \qquad
  \ell_\nu = \rk(\g_\nu,\kalg_\nu).
\]
Then
\[
  \dim\Cr(\p)
  =
  \sum_{\nu=1}^t
  \max\bigl\{
    rd_\nu,\,
    2d_\nu+(r-1)\ell_\nu
  \bigr\}.
\]
\end{corollary}

\begin{proof}
By \cref{lem:products}, the higher commuting variety respects direct sums:
\[
  \Cr(\p) \cong \prod_{\nu=1}^t\Cr(\p_\nu).
\]
The formula follows immediately from \cref{thm:full-hermitian} and the
additivity of Krull dimension under Cartesian products of complex algebraic
varieties.
\end{proof}

\begin{remark}\label{rem:full-explicit}
Substituting the explicit invariants \(d\) and \(\ell\) from the classification
into \cref{thm:full-hermitian} yields the following dimensions for each simple
Hermitian type:
\[
\begin{array}{@{}lll@{}}
\toprule
\text{type} & \text{ambient algebra } \g & \dim\Cr(\p) \\
\midrule
\mathrm{AIII}(m,n),\ 1\le m\le n & \mathfrak{sl}_{m+n} & \max\{rmn,\ 2mn+(r-1)m\} \\[1.5mm]
\mathrm{BDI}_{\mathrm{Herm}}(n),\ n\ge3 & \mathfrak{so}_{n+2} & \max\{rn,\ 2n+2(r-1)\} \\[1.5mm]
\mathrm{CI}(n),\ n\ge1 & \mathfrak{sp}_{2n} & \displaystyle\max\left\{ \frac{rn(n+1)}2,\, n(n+1)+(r-1)n \right\} \\[3mm]
\mathrm{DIII}(n),\ n\ge3 & \mathfrak{so}_{2n} & \displaystyle\max\left\{ \frac{rn(n-1)}2,\, n(n-1)+(r-1)\left\lfloor\frac n2\right\rfloor \right\} \\[3mm]
\mathrm{EIII} & \mathfrak{e}_6 & \max\{16r,\ 32+2(r-1)\} \\[1.5mm]
\mathrm{EVII} & \mathfrak{e}_7 & \max\{27r,\ 54+3(r-1)\} \\
\bottomrule
\end{array}
\]

For \(r=3\), the linear nilpotent families have dimension greater than or
equal to the Cartan family in every case except
\[
  (\mathfrak{sl}_2,\mathfrak{so}_2) \quad\text{and}\quad \mathrm{CI}(2).
\]
These dimensions coincide at \(r=3\) precisely for the three cases listed in
\cref{cor:full-components}(iii). For \(\mathrm{CI}(2)\),
\[
  \dim C_3(\p) = 10 > 9 = 3d;
\]
here, the two types of families tie at \(r=4\) (dimension \(12\)), and the
linear nilpotent families strictly dominate for all \(r\ge5\). For the
rank-one pair \((\mathfrak{sl}_2,\mathfrak{so}_2)\), the Cartan family strictly
dominates for every \(r\ge2\), giving \(\dim\Cr(\p) = r+1\).
\end{remark}

\begin{remark}\label{rem:nilpotent-vs-full}
Comparing \cref{thm:full-hermitian} with \cref{thm:hermitian}, we obtain
\[
  \dim\Cr(\p) = \dim\Cr(\N(\p)) = rd
\]
for every \(r\ge3\), with precisely the following exceptions:
\[
  (\mathfrak{sl}_2,\mathfrak{so}_2) \quad\text{for all } r\ge3,
  \qquad\text{and}\qquad
  \mathrm{CI}(2) \quad\text{at } r=3.
\]
Thus, with these two low-rank exceptions, removing the nilpotency condition
does not increase the dimension of the higher commuting variety once \(r\ge3\).
For pairs (\(r=2\)), by contrast, the semisimple contribution always strictly
dominates:
\[
  \dim C_2(\p) - \dim C_2(\N(\p)) = \ell = \rk(\g,\kalg).
\]
\end{remark}

\subsection{Fields of validity}\label{subsec:fields}
The characteristic-zero results are not restricted to the complex
numbers. The span construction has a still wider range of validity.

\begin{proposition}\label{prop:charzero} The dimension, irreducible-component, and isotropy statements
formulated over $\C$ in this paper hold over every algebraically
closed field of characteristic zero.
\end{proposition}

\begin{proof} For each statement, the symmetric pair, its grading when present,
and the relevant finite-type schemes, morphisms, and differential
forms descend to a subfield $F$ of the given field that is
finitely generated over $\Q$. 
We choose embeddings of an algebraic closure $\overline F$
into the given field and into $\C$.
Dimensions and geometric irreducible components are preserved by
extension of algebraically closed fields. Smooth loci and the
vanishing of the algebraic differential forms are compatible
with these extensions as well. The assertions therefore descend
from $\C$ to $\overline F$ and extend to the
given field.
\end{proof}

The span stratification and eventual dimension argument apply over
any algebraically closed field whenever the chosen nilpotence
condition is preserved by linear combinations of commuting elements.
This includes nilpotent matrix tuples. It also includes commuting $p$-nilpotent elements in a restricted Lie
algebra over a field of characteristic $p>0$, with $p$-operation
$x\mapsto x^{[p]}$: for commuting elements, \[
 \left(\sum_i a_i x_i\right)^{[p]}
 =\sum_i a_i^p x_i^{[p]}.
\]
Thus the restricted nullcone contains their linear span whenever
all $x_i^{[p]}$ vanish. 
One must distinguish the restricted nullcone
\[
 \N_1(\g)=\{x\in\g:x^{[p]}=0\}
\]
from the full nilpotent cone unless their equality has been established.

\begin{proposition}\label{prop:large-p}
Fix a simple Hermitian root datum and its three-grading.
There is an integer $N>0$, depending on this datum but independent
of $r$, such that for every $r\ge2$ and every algebraically closed
field $\kk$ of characteristic zero or characteristic $p\nmid N$,
\[
 \dim\Cr\bigl(\N(\p_\kk)\bigr)=r\dim\pp_\kk.
\]
\end{proposition}

\begin{proof}
We spread out the symmetric pair and its three-grading over
$R_0=\Z[1/N_0]$, inverting enough primes to obtain the required
eigenspace decompositions and a model of the nilpotent cone.
Write $d$ for the rank of the positive grading piece.
There is a finite-type $R_0$-scheme $\mathcal C_2$ whose
geometric fibres have reductions $C_2(\N(\p_\kk))$.
Its geometric generic fibre has dimension $2d$.

By generic constancy of fibre dimension
\cite[Tag~05F7]{Stacks}, after inverting finitely many further
primes every geometric fibre of $\mathcal C_2$ has dimension $2d$.
Fix the resulting ring $\Z[1/N]$.
For any allowed field and any $r\ge2$, every two-coordinate
projection of an irreducible component of
$\Cr(\N(\p_\kk))$ has dimension at most $2d$.
\Cref{lem:pair-projection} therefore gives the upper bound $rd$.
The integral three-grading supplies the abelian nilpotent
subspace $\pp_\kk$, so $(\pp_\kk)^r$ gives the matching lower
bound. The same integer $N$ works for every $r$.
\end{proof}

The integer $N$ is not effective in this argument. In particular,
the proposition does not assert validity in every good
characteristic, nor does it establish reducedness of the
unreduced commuting schemes.

\section{Examples and comparisons with known results}
\label{sec:examples}

We illustrate the main results in three representative situations.
The rank-one Hermitian pair makes the geometry of both the nilpotent
and full commuting varieties completely explicit. The group case
recovers the classical dimension theorem for pairs of commuting
nilpotent elements, while type~$A$ shows how the span stratification
explains the asymptotic formula of Levy--Ngo--\v{S}ivi\'c.

\subsection{Two rank-one examples} In the two smallest examples, both commuting varieties can be described
explicitly.

\begin{example}
\label{ex:sl2herm}
Consider $(\g,\kalg) = (\mathfrak{sl}_2,\mathfrak{so}_2)$,
the Hermitian symmetric pair of type $\mathrm{AIII}(1,1)$, where
$\kalg = \C h \cong \mathfrak{s}(\mathfrak{gl}_1\oplus\mathfrak{gl}_1)$ and
$\p = \C e\oplus\C f$ in terms of the standard $\mathfrak{sl}_2$-triple $(e,h,f)$.
The three-grading is given by
\[
  \pp = \C e, \qquad \pmns = \C f.
\]
For any two elements $x = ae+bf$ and $y = a'e+b'f$ in $\p$, their Lie bracket is
\[
  [x,y] = (ab'-a'b)h.
\]
Thus a tuple $(x_1,\dots,x_r)\in\p^r$, with $x_i = a_i e + b_i f$, commutes if
and only if
\[
  a_i b_j - a_j b_i = 0 \qquad \text{for all } 1\le i<j\le r.
\]
Consequently, the full commuting variety $\Cr(\p)$ is canonically isomorphic to
the determinantal variety of $2\times r$ matrices of rank at most $1$:
\[
  \Cr(\p)
  \cong
  \left\{
    \begin{pmatrix} a_1 & \cdots & a_r \\ b_1 & \cdots & b_r \end{pmatrix}
    \in \operatorname{Mat}_{2\times r}(\C) :
    \operatorname{rk}\begin{pmatrix} a_1 & \cdots & a_r \\ b_1 & \cdots & b_r \end{pmatrix} \le 1
  \right\}.
\]
Because the variety of $2\times r$ matrices of rank at most $1$ has dimension
$2 + r - 1 = r + 1$, we obtain
\[
  \dim\Cr(\p) = r+1,
\]
in exact agreement with \cref{thm:full-hermitian} (since $d=1$ and $\ell=1$,
giving $\max\{r,\, 2+(r-1)\} = r+1$).

On the other hand, an element $ae+bf \in \p$ has determinant $-ab$, so it is
nilpotent if and only if $ab=0$. It follows that a commuting tuple of
nilpotents must have all its entries in $\C e$ or all its entries in $\C f$.
Therefore,
\[
  \Cr(\N(\p)) = (\C e)^r \cup (\C f)^r = (\pp)^r \cup (\pmns)^r.
\]
Thus $\Cr(\N(\p))$ has dimension $r$ and has exactly two irreducible
components. Moreover, the variety of $1$-dimensional abelian nilpotent
subspaces is
\[
  \Ab_1(\p) = \{\pp,\,\pmns\},
\]
providing the smallest instance of \cref{prop:Ab-hermitian} and
\cref{thm:hermitian-max-components}.
\end{example}

\begin{example}
\label{ex:sl2group}
For the group-type symmetric pair
\[
  (\mathfrak{sl}_2\oplus\mathfrak{sl}_2,\, \Delta\mathfrak{sl}_2),
\]
we have $\p\cong\mathfrak{sl}_2$ and hence $\Cr(\N(\p))\cong\Cr(\N(\mathfrak{sl}_2))$.
In $\mathfrak{sl}_2$, every nonzero nilpotent element has a $1$-dimensional
centralizer. Consequently, two nonzero nilpotent elements commute if and only if
they are proportional. Therefore,
\[
  \Cr(\N(\mathfrak{sl}_2))
  =
  \left\{
    (t_1 x, \dots, t_r x) :
    x\in\N(\mathfrak{sl}_2),\ t_1,\dots,t_r\in\C
  \right\}.
\]
Since the projectivized null cone $\mathbb P\N(\mathfrak{sl}_2)$ is isomorphic to
a smooth quadric in $\mathbb P^2$ (hence is isomorphic to $\mathbb P^1$), the
variety of such tuples has dimension
\[
  \dim\Cr(\N(\mathfrak{sl}_2)) = \dim\mathbb P^1 + r = r+1.
\]
On the other hand, the universal estimate of \cref{thm:universal} gives
\[
  \dim\Cr(\N(\mathfrak{sl}_2))
  \le
  \left\lfloor\frac{3r}{2}\right\rfloor.
\]
Comparing the exact dimension $r+1$ with $\lfloor 3r/2 \rfloor$:
\begin{itemize}
\item For $r=2$: $r+1 = 3 = \lfloor 6/2 \rfloor$ (sharp).
\item For $r=3$: $r+1 = 4 = \lfloor 9/2 \rfloor$ (sharp).
\item For $r=4$: $r+1 = 5 < 6 = \lfloor 12/2 \rfloor$ (strict).
\end{itemize}
Thus the universal bound is sharp for $r\in\{2,3\}$ and strictly exceeds the
actual dimension for every $r\ge4$.
\end{example}

\subsection{The group case and Premet's theorem}
\label{subsec:groupcheck}

The group-type symmetric pair shows that the universal estimate specializes
directly to ordinary semisimple Lie algebras.

\begin{corollary}\label{cor:groupcase}
Let $\g_0$ be a complex semisimple Lie algebra. Then, for every $r\ge2$,
\[
  \dim\Cr(\N(\g_0))
  \le
  \left\lfloor
    \frac{r\dim\g_0}{2}
  \right\rfloor.
\]
\end{corollary}

\begin{proof}
Apply \cref{thm:universal} to the group-type symmetric pair
$(\g_0\oplus\g_0, \Delta\g_0)$ (\cref{ex:groupcase}), for which
$\p\cong\g_0$ and $\Cr(\N(\p))\cong\Cr(\N(\g_0))$.
\end{proof}

For $r=2$, Premet proved that $C_2(\N(\g_0))$ is equidimensional of dimension
$\dim\g_0$, with irreducible components indexed by the distinguished nilpotent
orbits \cite{Premet}. Thus \cref{cor:groupcase} is sharp for $r=2$.
More generally, \cref{prop:r2-nilpotent} gives
\[
  \dim C_2(\N(\p)) = \dim\p
\]
for every complex semisimple symmetric pair. Using the Kostant--Rallis formula
\eqref{eq:KR}, this may equivalently be written as
\[
  \dim C_2(\N(\p)) = \dim\N(\p) + \rk(\g,\kalg).
\]

\subsection{The span formula in type \texorpdfstring{$A$}{A}}
\label{subsec:asymptotic}

Throughout this subsection, let $n\ge4$ and put
$a_n=\lfloor n^2/4\rfloor$. Here $\N(\mathfrak{gl}_n)$ denotes
the cone of nilpotent matrices. We use $a(\mathfrak{sl}_n)$ and
$\Ab_e(\mathfrak{sl}_n)$ via the group-type identification in
\cref{ex:groupcase}. Since every nilpotent matrix is traceless,
\[
  \Cr(\N(\mathfrak{gl}_n)) = \Cr(\N(\mathfrak{sl}_n)).
\]
Levy, Ngo, and \v{S}ivi\'c proved, in characteristic different
from $2$ and $3$, that \cite[Theorem~4.8]{LevyNgoSivic}
\begin{equation}\label{eq:LNS-again}
  \dim\Cr(\N(\mathfrak{gl}_n)) = (r+1)a_n \qquad (r\ge7).
\end{equation}

The span stratification of \cref{subsec:span} explains the
conceptual structure of this formula. By the Schur--Mal'cev theorem
\cite{Schur,Malcev}, the maximal dimension of an abelian subalgebra of
$\mathfrak{sl}_n$ consisting of nilpotent elements is
\[
  a(\mathfrak{sl}_n) = a_n.
\]
For a linear subspace $W\subseteq\C^n$, define
\[
  \mathfrak m_W
  :=
  \left\{
    x\in\End(\C^n) : x(W)=0,\ x(\C^n)\subseteq W
  \right\}
  \cong
  \Hom(\C^n/W, W).
\]
For $x,y\in\mathfrak m_W$, the composition $xy$ is zero.
Thus $\mathfrak m_W$ is an abelian subalgebra consisting of nilpotent
endomorphisms, of dimension $\dim W(n-\dim W)$. 

Mal'cev's classification of maximal-dimensional abelian nilpotent subalgebras
\cite{Malcev} implies that every such subalgebra is of the form $\mathfrak m_W$
where $\dim W \in \{\lfloor n/2 \rfloor, \lceil n/2 \rceil\}$. Consequently,
the parameter space $\Ab_{a_n}(\mathfrak{sl}_n)$ is isomorphic to
\[
  \Ab_{a_n}(\mathfrak{sl}_n)
  \cong
  \begin{cases}
    \Gr(q, 2q), & \text{if } n=2q, \\[1mm]
    \Gr(q, 2q+1) \sqcup \Gr(q+1, 2q+1), & \text{if } n=2q+1.
  \end{cases}
\]
The inverse to $W\mapsto\mathfrak m_W$ is
$A\mapsto\sum_{x\in A}\operatorname{im}x$.
On each locus of fixed $\dim W$, the evaluation map
$\mathcal S_{a_n}\otimes\C^n\to\C^n\otimes\mathcal O$
has constant rank, so this inverse is a regular Grassmannian map.
The two families in odd dimension are disjoint closed images of
projective Grassmannians. This justifies the displayed isomorphism
of reduced varieties, beyond the set-theoretic classification.

Each connected component is a Grassmannian of dimension
$q(n-q) = a_n$. Therefore,
\[
  \dim\Ab_{a_n}(\mathfrak{sl}_n) = a_n.
\]
Applying our eventual-dimension theorem (\cref{thm:asymptotic}) to
$\p \simeq \mathfrak{sl}_n$, the asymptotic dimension is given by
\[
  \dim\Cr(\N(\mathfrak{sl}_n))
  =
  r\cdot a(\mathfrak{sl}_n) + \dim\Ab_{a(\mathfrak{sl}_n)}(\mathfrak{sl}_n)
  =
  ra_n + a_n
  =
  (r+1)a_n
\]
for all sufficiently large $r$.

Thus \cref{thm:asymptotic} recovers the structural form of \eqref{eq:LNS-again}
conceptually, identifying the slope as the Schur--Mal'cev invariant and the
intercept as the dimension of the parameter space of maximal abelian nilpotent
subspaces. The deeper contribution of Levy, Ngo, and \v{S}ivi\'c is establishing
the effective uniform threshold $r\ge7$ and classifying the
maximal-dimensional components.

By contrast, for a complex Hermitian symmetric pair, the corresponding
invariants are
\[
  a(\p) = \dim\pp = \frac12\dim\p,
  \qquad
  \dim\Ab_{a(\p)}(\p) = 0.
\]
Here, the general asymptotic expression $r\cdot a(\p) + \dim\Ab_{a(\p)}(\p)$
evaluates to $r\dim\pp + 0 = r\dim\pp$. In this case, \cref{thm:hermitian}
proves that this formula is not merely asymptotic, but holds immediately for
every $r\ge2$.

\section{Applications to punctual moduli spaces}
\label{sec:applications}

We apply the preceding results to moduli spaces of finite-length
sheaves and finite subschemes supported at one point. We first explain
these moduli spaces and their descriptions by commuting matrices.
The Hermitian pair of type $\mathrm{AIII}$ then gives dimension
and component results for equivariant Quot schemes. Finally, the
bilinear forms associated with Gorenstein algebras lead to the
symmetric pairs of types $\mathrm{AI}$, $\mathrm{CI}$, and
$\mathrm{BDI}$ and yield dimension bounds for Gorenstein loci.

Throughout this section, we work over $\C$ and put
\[
  R=\C[x_1,\ldots,x_r],
  \qquad
  \mathfrak m=(x_1,\ldots,x_r),
  \qquad r\geq2.
\]
We use the convention $\dim\varnothing=-\infty$. All dimension and
component statements below concern the reduced structures of the
indicated moduli spaces.

\subsection{Punctual Hilbert and Quot schemes}
\label{subsec:punctual-moduli}

For $N\ge1$, the Hilbert scheme of $N$ points, denoted
$\operatorname{Hilb}^N(\mathbb A^r)$, parametrises closed
zero-dimensional subschemes of $\mathbb A^r$ of length $N$.
Its closed points are equivalently ideals $I\subset R$ with
$\dim_\C R/I=N$. More generally, a family over a scheme $T$
is a closed subscheme of $\mathbb A^r\times T$ that is finite
and flat of degree $N$ over $T$.
The Hilbert--Chow morphism sends a finite subscheme to its
support cycle:
\[
 \operatorname{hc}\colon\operatorname{Hilb}^N(\mathbb A^r)
 \longrightarrow\operatorname{Sym}^N(\mathbb A^r),\qquad
 Z\longmapsto
 \sum_{z\in\operatorname{Supp}Z}
             \operatorname{length}(\mathcal O_{Z,z})[z].
\]
Here $\operatorname{Sym}^N(\mathbb A^r)$ is the symmetric product,
whose closed points are effective zero-cycles of degree $N$.
The punctual Hilbert scheme at the origin is the fibre
$\operatorname{hc}^{-1}(N[0])$; we write
$\operatorname{Hilb}^N_0(\mathbb A^r)$ for its reduction.
Its points are exactly the ideals $I$ of colength $N$ with
$\sqrt I=\mathfrak m$. Thus $R/I$ is a local Artinian algebra.

For $L\ge1$ and a finite-dimensional vector space $W$ of dimension $s$, the Quot scheme
$\operatorname{Quot}^L_{\mathbb A^r}
(\mathcal O_{\mathbb A^r}\otimes W)$ parametrises quotients
\[
 q\colon\mathcal O_{\mathbb A^r}\otimes W
       \twoheadrightarrow\mathcal F,
 \qquad \operatorname{length}(\mathcal F)=L.
\]
Two quotients are identified if an isomorphism of their targets
intertwines the quotient maps; equivalently, they have the same
kernel. Families are quotients flat over the parameter scheme,
with support finite over it and fibre length $L$.
The support-cycle morphism is
\[
 \operatorname{cyc}\colon
 \operatorname{Quot}^L_{\mathbb A^r}
 (\mathcal O_{\mathbb A^r}\otimes W)
 \longrightarrow\operatorname{Sym}^L(\mathbb A^r),
 \qquad
 [q]\longmapsto
 \sum_{z\in\operatorname{Supp}\mathcal F}
          \operatorname{length}(\mathcal F_z)[z].
\]
We write
$\operatorname{Quot}^L_0(\mathcal O_{\mathbb A^r}\otimes W)
:=\operatorname{cyc}^{-1}(L[0])_{\mathrm{red}}$
for the reduced punctual Quot scheme at the origin.
When $W=\C$, the kernel is an ideal and this is precisely
$\operatorname{Hilb}^L_0(\mathbb A^r)$.
For larger $W$, the target is an $R$-module and need not be an
algebra.

The universal quotient on a Quot scheme gives a rank-$L$
tautological vector bundle by taking its direct image to the
parameter scheme. Its fibre at $[q]$ is $H^0(\mathcal F)$.
For the Hilbert scheme, this is the tautological algebra bundle
with fibre $R/I$.

We recall the classical surface results to put the applications
in context. Fogarty proved that the Hilbert scheme of $N$ points
on a smooth surface is smooth of dimension $2N$
\cite[Theorem~2.4]{Fogarty}.
Brian\c{c}on proved that
$\operatorname{Hilb}^N_0(\mathbb A^2)$ is irreducible of
dimension $N-1$ \cite{Briancon}.
More generally, for $s,L\ge1$, the ordinary punctual Quot
scheme of a rank-$s$ bundle on a smooth surface is irreducible
of dimension $sL-1$; this is due independently to
Ellingsrud--Lehn and Baranovsky \cite[Theorem~1]{EllingsrudLehn} and 
\cite{BaranovskyQuot}.
The equivariant surface case will be related to Nakajima
varieties in \cref{rem:equivariant-quot-surface}.

The connection with commuting varieties is elementary.
A finite-length sheaf on $\mathbb A^r$ is the same as a
finite-dimensional $R$-module $M=H^0(\mathcal F)$.
After choosing a basis of $M$, multiplication by $x_i$ is an
endomorphism $X_i$; we write $X_\bullet=(X_1,\ldots,X_r)$.
The relations in $R$ say that
$X_iX_j=X_jX_i$, and support at the origin is equivalent
to nilpotence of every $X_i$.
If $\dim M=L$, simultaneous strict triangularisation gives
\[
 X_{i_1}\cdots X_{i_L}=0
 \quad\text{for all }i_1,\ldots,i_L,
 \qquad\text{equivalently}\qquad \mathfrak m^L M=0.
\]
Thus every closed-point module of length $L$ is a module over
$R/\mathfrak m^L$. This also describes the tautological families
on the reduced punctual moduli schemes used below. Indeed,
multiplication by a monomial of degree $L$ is a morphism of
tautological vector bundles that vanishes at every closed point;
over a reduced scheme of finite type over $\C$, it therefore
vanishes identically. All subsequent identifications with
moduli of $R/\mathfrak m^L$-modules are made after reduction.

A quotient $R\otimes W\twoheadrightarrow M$ is specified by
a linear map $\phi\colon W\to M$ whose image generates $M$
as an $R$-module. We call $\phi$ a generating framing.
When $\dim W=1$, this is a cyclic vector
$v\in M$, meaning $\C[X_1,\ldots,X_r]v=M$.
Changing the basis of $M$ gives the frame bundle of the
tautological bundle. This construction, including its
scheme-theoretic form, is standard
\cite[Proposition~3.7 and \S3.4]{JelisiejewSivic}.
Our applications use its restriction to the reduced punctual
and equivariant loci.

Let $\mu_2=\{\pm1\}$ act on $\mathbb A^r$ by
$(x_1,\ldots,x_r)\mapsto(-x_1,\ldots,-x_r)$.
An equivariant module is a module $M$ with an involution
$\varepsilon$ satisfying
$\varepsilon(x_i u)=-x_i\varepsilon(u)$.
It decomposes as $M=M_+\oplus M_-$ into the two eigenspaces
of $\varepsilon$, and each $X_i$ interchanges these spaces.
We call the two characters even and odd, respectively.
For an invariant ideal $I$, the quotient $R/I$ inherits
this decomposition and its unit is even.
Invariance refers only to this parity decomposition; it
does not require homogeneity in the usual integer grading.

\subsection{Equivariant punctual sheaves and Quot schemes}
\label{subsec:mu2-AIII}

We first fix the dimensions of the two character spaces.
The resulting matrix problem is the nilpotent commuting
variety for a symmetric pair of type $\mathrm{AIII}$.

Fix vector spaces $V_+$ and $V_-$
of dimensions $m$ and $n$, where $m,n\geq0$ and $m+n\geq1$, and set
\[
  V=V_+\oplus V_-,
  \qquad
  G_{m,n}=\mathrm{GL}(V_+)\times\mathrm{GL}(V_-).
\]
The odd endomorphisms of $V$ form the space
\[
  \mathfrak p_{m,n}
  =\mathfrak p_{m,n}^+\oplus\mathfrak p_{m,n}^-,
  \qquad
  \mathfrak p_{m,n}^+=\operatorname{Hom}(V_-,V_+),
  \qquad
  \mathfrak p_{m,n}^-=\operatorname{Hom}(V_+,V_-),
\]
where each block is extended by zero on the other summand.
For $m,n\geq1$, this is the isotropy representation of the
Hermitian symmetric pair
\[
  \bigl(
    \mathfrak{sl}(V),
    \mathfrak s(\mathfrak{gl}(V_+)\oplus\mathfrak{gl}(V_-))
  \bigr)
\]
of type $\mathrm{AIII}$, with
$\dim\mathfrak p_{m,n}^+=mn$.
Here $\mathfrak s(\mathfrak{gl}(V_+)\oplus
\mathfrak{gl}(V_-))$ consists of block-diagonal endomorphisms
whose total trace is zero.
If either multiplicity is zero, then $\mathfrak p_{m,n}=0$.

An equivariant $R$-module structure on $V$ supported at the origin
is therefore a point of
\[
  \mathcal C_{m,n}
  :=\Cr\bigl(\N(\mathfrak p_{m,n})\bigr)
  =\{X_\bullet\in\mathfrak p_{m,n}^{\,r}:
       [X_i,X_j]=0\text{ for all }i,j,
       \ X_i\text{ nilpotent for all }i\},
\]
with its reduced structure.
Here $\N(\mathfrak p_{m,n})$ denotes the nilpotent
endomorphisms in $\mathfrak p_{m,n}$.
Indeed, equivariance makes the coordinate operators odd, while support
at the origin is equivalent to their nilpotence. For commuting
operators, individual nilpotence implies that the ideal they generate
acts nilpotently.

\begin{remark}
\label{rem:equivariant-stack}
Put $L=m+n$. Let $\mathfrak M^{m,n}_{r,0}$ be the reduction
of the moduli stack of $\mu_2$-equivariant
$R/\mathfrak m^L$-modules whose even and odd parts are locally
free of ranks $m$ and $n$. Its closed points are the
equivariant coherent sheaves supported at the origin with
these character multiplicities. Choosing frames of the two
character bundles gives
\[
 \mathfrak M^{m,n}_{r,0}
 \simeq[\mathcal C_{m,n}/G_{m,n}],
 \qquad
 \dim\mathfrak M^{m,n}_{r,0}=rmn-m^2-n^2.
\]
Here $[X/H]$ denotes the quotient stack, which retains
stabilisers and has dimension $\dim X-\dim H$.
The formula follows from \cref{thm:hermitian} when $m,n\ge1$;
if one multiplicity is zero, then $\mathcal C_{m,n}$ is a point.

For $m,n\ge1$, the images of
$(\mathfrak p_{m,n}^+)^r$ and $(\mathfrak p_{m,n}^-)^r$
are distinct maximal-dimensional irreducible components. They parametrise modules $M$ with
$\mathfrak mM\subseteq M_+$ and
$\mathfrak mM\subseteq M_-$, respectively. Indeed, each $X_i$ is odd,
so $\mathfrak mM\subseteq M_+$ holds exactly when every $X_i$ kills $M_+$,
that is, when $X_\bullet\in(\mathfrak p_{m,n}^+)^r$; then
$\mathfrak m^2M\subseteq\mathfrak mM_+=0$. The other case is symmetric.
Since $G_{m,n}$ is connected, the quotient preserves the
component correspondence. For $r\ge3$, these are the only
maximal-dimensional components by
\cref{thm:hermitian-max-components}; no such exhaustiveness
is asserted here for $r=2$.
\end{remark}

For an equivariant sheaf $\mathcal F$, we write
$H^0(\mathcal F)_\pm$ for the eigenspaces of the involution with eigenvalues
$\pm1$.
We next add a framing. Let $W=W_+\oplus W_-$ with
$\dim W_\pm=s_\pm$, where $s_+,s_-\geq0$.
We write $\mathsf Q^{m,n;s_+,s_-}_{r,0}$ for the reduced scheme
of equivariant quotients
\[
  \mathcal O_{\mathbb A^r}\otimes W
  \twoheadrightarrow\mathcal F
\]
supported at the origin and satisfying
\[
  \dim H^0(\mathcal F)_+=m,
  \qquad
  \dim H^0(\mathcal F)_-=n.
\]
Equivalently, this is the reduction of the open-and-closed
part of the $\mu_2$-fixed locus in the punctual Quot scheme
on which the two character ranks are $m$ and $n$.
The action on the source uses the indicated decomposition of $W$.
An invariant kernel inherits a unique quotient linearisation
making the quotient map equivariant.
In particular, for the trivial one-dimensional framing,
\[
 \mathsf Q^{m,n;1,0}_{r,0}
 =
 \bigl(\operatorname{Hilb}^{m+n}_0(\mathbb A^r)^{\mu_2}
                   _{(m,n)}\bigr)_{\mathrm{red}},
\]
where the subscript $(m,n)$ specifies the dimensions of the
even and odd parts of $R/I$.

The dimension bound below is obtained by adding the framing
parameters to the matrix variety and subtracting the parameters
for a change of basis. Equality requires the generating locus
to meet a maximal-dimensional matrix component.

\begin{theorem}\label{thm:equivariant-quot}
Let $r\ge2$, $m,n,s_+,s_-\ge0$, and $m+n\ge1$.
Then
\begin{equation}\label{eq:equiv-quot-bound}
 \dim\mathsf Q^{m,n;s_+,s_-}_{r,0}
 \le rmn+s_+m+s_-n-m^2-n^2.
\end{equation}
Equality holds if at least one of the following conditions holds:
\begin{align}
 &s_-\ge n,\qquad m\le s_++rn,
 \label{eq:plus-stability}\\
 &s_+\ge m,\qquad n\le s_-+rm.
 \label{eq:minus-stability}
\end{align}
If $r\ge3$ and $m,n\ge1$, these conditions are also necessary:
equality holds if and only if at least one holds.
When equality holds in this range, there is exactly one
maximal-dimensional component for each of
\eqref{eq:plus-stability} and \eqref{eq:minus-stability}
that holds.

For every $r\ge2$, if $m,n\ge1$, $s_+\ge m$, and $s_-\ge n$,
the scheme has at least two maximal-dimensional components, two of which are the closed loci
\[
 \mathfrak mH^0(\mathcal F)\subseteq H^0(\mathcal F)_+,
 \qquad
 \mathfrak mH^0(\mathcal F)\subseteq H^0(\mathcal F)_-.
\]
For $r\ge3$ there are no others; for $r=2$ there may be
(\cref{rem:equivariant-quot-surface}).
\end{theorem}

\begin{proof}
Put
\[
 F=\Hom(W_+,V_+)\oplus\Hom(W_-,V_-),\qquad
 \dim F=s_+m+s_-n.
\]
The generating locus
\[
 \mathcal U
 =\{(X_\bullet,\phi)\in\mathcal C_{m,n}\times F:
                              R\phi(W)=V\}
\]
is open. To see this, put $L=m+n$. Simultaneous strict
triangularisation implies that every product of $L$ coordinate
operators is zero. Generation is therefore equivalent to
surjectivity of the finite matrix whose columns are
$X_1^{\alpha_1}\cdots X_r^{\alpha_r}\phi(w)$, for basis
vectors $w$ in $W$ and nonnegative integers $\alpha_i$ with
$\sum_i\alpha_i<L$. This is an open rank condition.
The map from $\mathcal U$ to $\mathsf Q^{m,n;s_+,s_-}_{r,0}$ is the
product of the frame bundles of the two tautological character
bundles, hence a Zariski locally trivial $G_{m,n}$-torsor.
The action is
\[
 g\cdot(X_\bullet,\phi)
   =\bigl((gX_ig^{-1})_{i=1}^r,g\phi\bigr)
 \qquad(g\in G_{m,n}).
\]
An automorphism commuting
with the $X_i$ and fixing $\phi(W)$ pointwise fixes its entire $R$-span,
so the action is free. Thus
\[
 \dim\mathsf Q^{m,n;s_+,s_-}_{r,0}
 =\dim\mathcal U-m^2-n^2.
\]
Since $\dim\mathcal C_{m,n}=rmn$, this proves
\eqref{eq:equiv-quot-bound}.

Write $\phi_\pm:W_\pm\to V_\pm$ for the two restrictions
of the framing. On $(\mathfrak p_{m,n}^+)^r$, all operators kill $V_+$.
The framing generates precisely when
\[
 \phi_-(W_-)=V_-,
 \qquad
 \phi_+(W_+)+\sum_iX_i(V_-)=V_+.
\]
Such a framing and tuple exist if and only if
$s_-\ge n$ and $s_++rn\ge m$:
after choosing a surjection $W_-\to V_-$, the second condition
asks for a surjection $W_+\oplus V_-^{\oplus r}\to V_+$.
This proves the equivalence with \eqref{eq:plus-stability}.
The analogous calculation on $(\mathfrak p_{m,n}^-)^r$
gives \eqref{eq:minus-stability}.
Each nonempty intersection is open in a variety of dimension
$rmn+\dim F$, so either condition implies equality in the bound.

Now assume $r\ge3$ and $m,n\ge1$.
The symmetric pair of type $\mathrm{AIII}$ has simple ambient
Lie algebra $\mathfrak{sl}_{m+n}$ and $d=mn$.
By \cref{thm:hermitian-max-components}, the only
irreducible components of $\mathcal C_{m,n}$ of dimension $rmn$
are
\[
 Z_+=(\mathfrak p_{m,n}^+)^r,\qquad
 Z_-=(\mathfrak p_{m,n}^-)^r.
\]
Therefore $\dim\mathcal U=rmn+\dim F$ if and only if
$\mathcal U$ meets $(Z_+\times F)\cup(Z_-\times F)$.
The preceding calculation identifies these two possibilities
with \eqref{eq:plus-stability} and \eqref{eq:minus-stability}.
Each nonempty intersection gives one maximal-dimensional
component, and no component of smaller ambient dimension can
give another. The frame torsor has connected irreducible fibre
$G_{m,n}$ and is locally trivial, so it induces the same
component correspondence on the Quot scheme.
This proves necessity and the component count.

Finally, for any $r\ge2$, the hypotheses $m,n\ge1$,
$s_+\ge m$, and $s_-\ge n$ ensure that both intersections
are nonempty. The two linear components are distinct, and, by \cref{rem:equivariant-stack}, their images under the torsor are exactly the two stated loci. 
These loci are closed because the conditions are the vanishing
of the induced maps between the tautological character bundles.
\end{proof}

\begin{remark}\label{rem:equivariant-quot-degenerate}
The assumptions $m,n\ge1$ in the component count are essential.
If $n=0$ and $s_+\ge m$, then the coordinate operators vanish,
and taking the kernel of the framing gives
\[
 \mathsf Q^{m,0;s_+,s_-}_{r,0}\cong\Gr(s_+-m,W_+).
\]
Here $\Gr(k,W)$ denotes the Grassmannian of $k$-dimensional
linear subspaces of $W$. The Quot scheme is empty when $s_+<m$.
The analogous statement holds
when $m=0$, with the two parities interchanged.
In these cases the two linear families coincide.
\end{remark}

\begin{remark}\label{rem:equivariant-quot-surface}
For $r=2$, the preceding construction has a standard quiver
interpretation. Write
\[
  X_1=\begin{pmatrix}0&A_1\\B_1&0\end{pmatrix},\qquad
 X_2=\begin{pmatrix}0&A_2\\B_2&0\end{pmatrix}.
\]
The equation $[X_1,X_2]=0$ is equivalent to
$A_1B_2-A_2B_1=0$ and $B_1A_2-B_2A_1=0$.
The Kronecker quiver has two vertices and two arrows from
one vertex to the other; its double adds a reverse arrow
for each arrow. Assigning $A_1,A_2$ to the two forward arrows
and $B_2,-B_1$ to the reverse arrows turns the displayed equations
into the preprojective relations.

Consider the Nakajima quiver variety for this doubled quiver,
with dimension vector $(m,n)$, framing vector $(s_+,s_-)$,
and complex moment-map parameter zero. Choose the generic
stability condition requiring the images of the framing maps
$W_\pm\to V_\pm$ to generate the representation.
Its Lagrangian fibre is the reduced fibre over zero of the
projective morphism from the stable quiver variety to its
affine quotient. By \cite[Lemma~5.9]{NakajimaALE}, transposed to the stability convention just specified, this fibre consists
of jointly nilpotent representations with the reverse framing
maps $V_\pm\to W_\pm$ equal to zero. It is therefore
isomorphic to $\mathsf Q^{m,n;s_+,s_-}_{2,0}$.
Nakajima's theorem \cite[Theorem~5.8]{NakajimaALE} consequently
gives the stronger known statement that every nonempty
$\mathsf Q^{m,n;s_+,s_-}_{2,0}$ is equidimensional of dimension
\[
  2mn+s_+m+s_-n-m^2-n^2.
\]
In particular, when $r=2$, nonemptiness alone implies equality
in \eqref{eq:equiv-quot-bound}. The conditions
\eqref{eq:plus-stability} and \eqref{eq:minus-stability}
are sufficient but need not be necessary in this case.
For example, the invariant quotient $\C[x,y]/(x^4,y)$ has
character multiplicities $(2,2)$ and trivial one-dimensional
framing, although neither condition holds.
The necessity and exact component count in
\cref{thm:equivariant-quot} concern $r\ge3$.
\end{remark}

\subsection{Gorenstein loci and socle parity}
\label{subsec:gorenstein-applications}

We now pass from arbitrary quotient modules to quotient
algebras with a duality property. For a local quotient
$A=R/I$ supported at the origin, let
\[
 \mathfrak m_A=\mathfrak m A,\qquad
 \operatorname{soc}(A)
 =\operatorname{Ann}_A(\mathfrak m_A)
 =\{a\in A:\mathfrak m_Aa=0\}.
\]
The last space is the socle of $A$.
A local Artinian algebra is Gorenstein if and only if its
socle is one-dimensional. We write
\[
 \mathcal H_{r,N}
 :=
 \operatorname{Hilb}^{N,\mathrm{Gor}}_0(\mathbb A^r)
 =
 \{[I]\in\operatorname{Hilb}^N_0(\mathbb A^r):
                  \dim_\C\operatorname{soc}(R/I)=1\}.
\]
We equip this locus with its reduced induced structure.
It is open in the punctual Hilbert scheme: in a finite flat
family, the socle is the kernel of the multiplication map
$a\mapsto(x_1a,\ldots,x_ra)$, and the condition that this
kernel have its minimum possible dimension, namely one,
is open. This locus parametrises embedded quotient algebras
$R/I$, rather than abstract algebras up to isomorphism. The broader geometry of Gorenstein loci in the full Hilbert
scheme, including smoothability and irreducibility, is studied
in \cite{CasnatiJelisiejewNotari}. Those results concern loci
whose support can vary, whereas we fix the support at the origin.

A \emph{Frobenius functional} on $A$ is a linear map
$\lambda:A\to\C$ for which the symmetric bilinear form
\[
 b_\lambda(u,w):=\lambda(uw)
\]
is nondegenerate. For a local Gorenstein algebra, this is
equivalent to $\lambda|_{\operatorname{soc}(A)}\ne0$:
the radical of $b_\lambda$ is an ideal, and every nonzero
ideal meets the one-dimensional socle. Conversely, if
$\lambda$ vanishes on the socle, then every socle element
lies in the radical. Thus the Frobenius functionals form
the complement of a hyperplane in $A^*$.
The Frobenius pairing is classical; see
\cite[Chapter~1]{IarrobinoKanev} and
\cite[Example~2.1]{JelisiejewIarrobino}.
Jelisiejew gives a matrix presentation of the oriented Gorenstein
locus, where the self-duality is considered up to a nonzero scalar
\cite[Proposition~1.7]{JelisiejewIarrobino}.
In the dimension counts below, we retain the functional itself,
rather than its scalar class. 

Suppose now that $A=R/I$ is a $\mu_2$-invariant Gorenstein
quotient, and let $\varepsilon$ be its induced involution.
Its one-dimensional socle is preserved by $\varepsilon$;
we call it \emph{even} or \emph{odd} according as
$\varepsilon$ acts on it by $+1$ or $-1$.
A functional is even or odd if
$\lambda\circ\varepsilon=\lambda$ or
$\lambda\circ\varepsilon=-\lambda$, respectively.
The two projections are
\[
 \lambda_+=\tfrac12(\lambda+\lambda\circ\varepsilon),
 \qquad
 \lambda_-=\tfrac12(\lambda-\lambda\circ\varepsilon).
\]
If $\lambda$ is Frobenius, the part with the same parity as the socle
remains nonzero there and is therefore Frobenius. An odd functional
vanishes on $A_+$, and an even one on $A_-$. Since $A_+A_+$ and $A_-A_-$
lie in $A_+$, an odd Frobenius functional makes $A_+$ and $A_-$
isotropic and pairs them perfectly, so they have equal dimension; an even
Frobenius functional makes them orthogonal and separately nondegenerate.

For $n\ge1$, let $\mathsf G^{\mathrm{odd}}_{r,2n}$ be the
reduced locus in $\mathcal H_{r,2n}^{\mu_2}$ with odd socle;
its character multiplicities are necessarily $(n,n)$.
For $m,n\ge0$ with $m+n\ge1$, let
$\mathsf G^{\mathrm{even}}_{r,(m,n)}$ be the reduced locus
in $\mathcal H_{r,m+n}^{\mu_2}$ with even socle and character
multiplicities $(m,n)$.
These loci are open and closed in the corresponding reduced
fixed Gorenstein locus, hence locally closed in the punctual
Hilbert scheme. Indeed, the character ranks are locally
constant. Moreover, on the reduced Gorenstein locus the
tautological map $a\mapsto(x_1a,\ldots,x_ra)$ has constant
kernel dimension one, so its kernel is the socle line bundle;
on the fixed locus its character is locally constant.
All algebras here are ordinarily commutative; no
supercommutative sign rule is imposed.

The following estimates arise from symmetric endomorphisms,
odd symplectic endomorphisms, and odd orthogonal endomorphisms,
respectively. The cyclic locus need not meet a
maximal-dimensional component of the ambient commuting variety,
so the estimates need not be equalities.

\begin{theorem}\label{thm:equivariant-gorenstein}
Let $r\ge2$. For $N\geq2$,
\begin{equation}\label{eq:gorenstein-bound}
  \dim\mathcal H_{r,N}
  \leq
  \left\lfloor
    \frac r2\left(\frac{N(N+1)}2-1\right)
  \right\rfloor
  -\frac{N(N-1)}2.
\end{equation}
For $n\geq1$,
\begin{equation}\label{eq:odd-gorenstein-bound}
  \dim\mathsf G^{\mathrm{odd}}_{r,2n}
  \leq\frac{rn(n+1)}2-n^2.
\end{equation}
For $m\geq2$ and $n\geq1$,
\begin{equation}\label{eq:even-gorenstein-bound}
  \dim\mathsf G^{\mathrm{even}}_{r,(m,n)}
  \leq
  \left\lfloor\frac{rmn}{2}\right\rfloor
  -\binom m2-\binom n2.
\end{equation}
The ordinary estimate also holds at any smooth point of a complex
$r$-dimensional variety. The hypotheses in
\eqref{eq:even-gorenstein-bound} are not restrictive: the length-one
Gorenstein locus is a point, and every invariant even-socle quotient of
length greater than one has $m\ge2$ and $n\ge1$, as the proof shows.
\end{theorem}

\begin{proof}
\emph{The ordinary locus.}
Fix a nondegenerate symmetric form $b$ on $V=\C^N$.
Write $\mathrm O(V,b)$ for its isometry group, consisting
of $g\in\mathrm{GL}(V)$ such that $b(gu,gw)=b(u,w)$ for
all $u,w\in V$. Put
\[
  \mathfrak p_{\mathrm{AI}}
  =\{X\in\mathfrak{sl}(V):b(Xu,w)=b(u,Xw)
                                      \text{ for all }u,w\in V\}.
\]
This is the $(-1)$-eigenspace of the symmetric pair
$(\mathfrak{sl}(V),\mathfrak{so}(V,b))$, and
\[
  \dim\mathfrak p_{\mathrm{AI}}=\frac{N(N+1)}2-1.
\]
Consider the open cyclic locus
\[
  \mathcal U_{\mathrm{AI}}
  =\{(X_\bullet,v)\in
       \Cr(\N(\mathfrak p_{\mathrm{AI}}))\times V:
       \C[X_1,\ldots,X_r]v=V\}.
\]
For $(X_\bullet,v)\in\mathcal U_{\mathrm{AI}}$, the kernel
$I$ of $R\to V$, $f\mapsto f(X_\bullet)v$, is an ideal:
if $f(X_\bullet)v=0$, commutativity gives
$(gf)(X_\bullet)v=g(X_\bullet)f(X_\bullet)v=0$.
Cyclicity identifies $A=R/I$ with $V$ as an $R$-module,
and the class of $1$ corresponds to $v$.
Since the $X_i$ are commuting nilpotent operators, $A$ is
supported at the origin. Define
\[
  \lambda(\bar f)=b\bigl(v,f(X_\bullet)v\bigr).
\]
Self-adjointness and commutativity imply
\[
  \lambda(\bar f\bar g)
  =b\bigl(f(X_\bullet)v,g(X_\bullet)v\bigr).
\]
Thus the Frobenius pairing on $A$ is identified with $b$,
so it is nondegenerate and $A$ is Gorenstein.
Conversely, a quotient $A=R/I$ with a Frobenius functional
$\lambda$ and an isometry $(A,b_\lambda)\simeq(V,b)$
gives such a tuple, with $v$ the image of $1$.
The multiplication operators are nilpotent and hence traceless,
so they do lie in $\mathfrak p_{\mathrm{AI}}$.

Over $\mathcal H_{r,N}$, the Frobenius functionals form an
open subset of the dual of the rank-$N$ tautological bundle,
with every fibre nonempty and of dimension $N$.
Over this space, the isometries form an orthogonal frame
torsor, locally trivial in the \'etale topology, with relative
dimension $\dim\mathrm O(V,b)=N(N-1)/2$.
The functional and the frame can be recovered from
$(X_\bullet,v)$ by the displayed formulas. The same
constructions use the tautological multiplication maps in
families and identify this frame torsor, after reduction,
with $\mathcal U_{\mathrm{AI}}$.
Its group action is free: an automorphism commuting with all
$X_i$ and fixing $v$ fixes every $f(X_\bullet)v$ and hence
all of $V$. Consequently,
\[
  \dim\mathcal U_{\mathrm{AI}}
  =\dim\mathcal H_{r,N}+N+\frac{N(N-1)}2.
\]
On the other hand, \cref{thm:universal} gives
\[
  \dim\mathcal U_{\mathrm{AI}}
  \le\left\lfloor\frac r2\dim\mathfrak p_{\mathrm{AI}}\right\rfloor+N.
\]
Subtracting proves \eqref{eq:gorenstein-bound}.
At a smooth point of an $r$-dimensional complex variety,
the completed local ring is $\C[[x_1,\ldots,x_r]]$.
Every local quotient of length $N$ is annihilated by the
$N$th power of the maximal ideal. The reduced punctual
length-$N$ loci therefore depend only on this finite
truncation, which gives the same estimate in formal local
coordinates.

\emph{Odd socle.} Let $A=R/I=A_+\oplus A_-$ be an invariant Gorenstein quotient with odd
socle, choose an odd Frobenius functional $\lambda$, and write
$b_\lambda(u,w)=\lambda(uw)$. Let $\varepsilon$ be the grading
involution. The form
\[
  \omega(u,w):=b_\lambda(\varepsilon u,w)
\]
is alternating and nondegenerate. Indeed, $\varepsilon$ is an algebra
automorphism and $\lambda\circ\varepsilon=-\lambda$, so
\[
  b_\lambda(\varepsilon u,w)
  =-\lambda\bigl(\varepsilon(\varepsilon(u)w)\bigr)
  =-b_\lambda(u,\varepsilon w),
\]
which proves skew-symmetry; composition with the invertible map
$\varepsilon$ preserves nondegeneracy.
Both $A_+$ and $A_-$ are isotropic of half the dimension,
so they are Lagrangian subspaces of dimension $n$.

Each coordinate multiplication $X_i$ is odd and $b_\lambda$-self-adjoint.
Therefore
\[
  \omega(X_i u,w)
  =-b_\lambda(X_i\varepsilon u,w)
  =-b_\lambda(\varepsilon u,X_i w)
  =-\omega(u,X_i w).
\]
Thus the $X_i$ lie in the odd part
$\mathfrak p_{\mathrm{CI}}$ of
\[
  \bigl(
    \mathfrak{sp}(A,\omega),
    \mathfrak{gl}(A_+)
  \bigr),
\]
the Hermitian pair of type $\mathrm{CI}$. Its positive and
negative grading pieces are
\[
  \mathfrak p_{\mathrm{CI}}^{+}
  =\mathfrak{sp}(A,\omega)\cap\operatorname{Hom}(A_-,A_+),
  \qquad
  \mathfrak p_{\mathrm{CI}}^{-}
  =\mathfrak{sp}(A,\omega)\cap\operatorname{Hom}(A_+,A_-),
\]
where each block is extended by zero on the other summand, and
\[
  \dim\mathfrak p_{\mathrm{CI}}^+
  =\frac{n(n+1)}2.
\]
Since $\omega(\varepsilon u,\varepsilon w)=-\omega(u,w)$, conjugation by
$\varepsilon$ preserves $\mathfrak{sp}(A,\omega)$ and is an involution
of it. Its fixed algebra consists of
the endomorphisms preserving the two Lagrangian summands;
the action on $A_+$ determines the negative dual action on
$A_-$. This identifies the fixed algebra with
$\mathfrak{gl}(A_+)$, as required for type $\mathrm{CI}$.

Conversely, on a fixed symplectic space $V=V_+\oplus V_-$
with $n$-dimensional Lagrangian summands, let
$\varepsilon|_{V_\pm}=\pm\mathrm{id}$. We now use the notation
\[
 \mathfrak p_{\mathrm{CI}}
 :=\{X\in\mathfrak{sp}(V,\omega):\varepsilon X=-X\varepsilon\}
\]
for the corresponding space on this fixed vector space and $\mathfrak p_{\mathrm{CI}}^{\pm}$ are defined as above, with $V_\pm$ in place of $A_\pm$.
Take a commuting nilpotent tuple in $\mathfrak p_{\mathrm{CI}}$
and an even cyclic vector $v\in V_+$.
The symmetric form
\[
  b(u,w)=\omega(\varepsilon u,w)
\]
makes every $X_i$ self-adjoint. Evaluation at $v$ identifies
$A=\C[X_1,\ldots,X_r]$ with $V$: an element annihilating $v$
annihilates its cyclic span and is therefore zero.
Equivalently, $A=R/I$ for the kernel of
$f\mapsto f(X_\bullet)v$.
This kernel is invariant under $\mu_2$ because the operators
are odd and $v$ is even. The functional
\[
  \lambda(a)=b(v,av)
\]
is Frobenius: self-adjointness gives $\lambda(aa')=b(av,a'v)$ for $a,a'\in A$, which
is nondegenerate under the identification $A\simeq V$. It is odd because
$b(V_+,V_+)=0$ and $v\in V_+$.

It follows that the even cyclic locus
\[
  \mathcal U_{\mathrm{odd}}
  =
  \left\{
  (X_\bullet,v)
  \in
  \Cr(\N(\mathfrak p_{\mathrm{CI}}))\times V_+
  \ \middle|\
  \C[X_1,\ldots,X_r]v=V
  \right\}
\]
parametrises these algebras with an odd Frobenius functional and a
graded symplectic frame, meaning a symplectic isomorphism
preserving the two character spaces. A change of frame on
$V_+$ determines the inverse dual change on $V_-$.
The frame group is therefore
$\mathrm{GL}(V_+)$, of dimension $n^2$, and the odd Frobenius
functionals form an open subset of $(A_-)^*$, of dimension $n$.
Hence
\[
  \dim\mathcal U_{\mathrm{odd}}
  =
  \dim\mathsf G^{\mathrm{odd}}_{r,2n}+n+n^2.
\]
By \cref{thm:hermitian}, the ambient commuting variety has
dimension $rn(n+1)/2$. Bounding the cyclic locus by its ambient
product and subtracting $n+n^2$ proves
\eqref{eq:odd-gorenstein-bound}.

\emph{Even socle.}
First, the even-socle locus with $m=1$ and $n\ge1$ is empty.
Indeed, its nonzero socle would equal $A_+=\C\cdot1$,
forcing $\mathfrak m_A=0$ and $A=\C$, contrary to $n\ge1$.
If $n=0$, every coordinate multiplication is zero, so the
only quotient algebra is $A=\C$, with $(m,n)=(1,0)$.
Thus every even-socle quotient of length greater than one
has $m\ge2$ and $n\ge1$, as in the statement.

We choose an even Frobenius functional and put
$b_\lambda(u,w)=\lambda(uw)$. Its character spaces are
orthogonal and separately nondegenerate.
We fix $V=V_+\oplus V_-$ with dimensions $m,n$ and a
nondegenerate symmetric form $b$ making the summands
orthogonal. A graded isometry $(A,b_\lambda)\simeq(V,b)$
transports the multiplication operators to $V$.
Let $\varepsilon|_{V_\pm}=\pm\mathrm{id}$.
For the coordinate multiplications $X_i$, set
\[
  Y_i=\varepsilon X_i.
\]
Oddness and self-adjointness of the $X_i$ imply
\[
  b(Y_i u,w)=-b(u,Y_i w),
  \qquad
  Y_iY_j=-X_iX_j.
\]
In particular, $Y_i^{2q}=(-1)^qX_i^{2q}$ for $q\ge1$,
so nilpotence is preserved. The $Y_i$ form a commuting nilpotent tuple in the
odd part
\[
 \mathfrak p_{\mathrm{BDI}}
 :=\{Y\in\mathfrak{so}(V,b):\varepsilon Y=-Y\varepsilon\}
\]
of the symmetric pair
\[
  \bigl(
    \mathfrak{so}(V,b),
    \mathfrak{so}(V_+)\oplus\mathfrak{so}(V_-)
  \bigr),
  \qquad
  \dim\mathfrak p_{\mathrm{BDI}}=mn.
\]

The transformation is inverted by $X_i=\varepsilon Y_i$.
For an even vector $v$ and $q\ge0$, we have
\[
 Y_{i_1}\cdots Y_{i_q}v
 =(-1)^{q(q+1)/2}X_{i_1}\cdots X_{i_q}v.
\]
Thus the two cyclic spans coincide.
Conversely, start with a commuting nilpotent tuple
$Y_\bullet$ in $\mathfrak p_{\mathrm{BDI}}$ and an even
cyclic vector $v$. The endomorphisms $X_i=\varepsilon Y_i$
are commuting, nilpotent, odd, and $b$-self-adjoint.
The kernel of $f\mapsto f(X_\bullet)v$ is an invariant
ideal $I$, and the resulting algebra $A=R/I$ is identified
with $V$.
The functional $\lambda(a)=b(v,av)$ is Frobenius by
self-adjointness and cyclicity. It is even because
$b(V_+,V_-)=0$, so the socle is even.

Thus the open locus
\[
 \mathcal U_{\mathrm{even}}
 :=\{(Y_\bullet,v)\in
       \Cr(\N(\mathfrak p_{\mathrm{BDI}}))\times V_+:
       \C[Y_1,\ldots,Y_r]v=V\}
\]
parametrises the even-socle quotients with an even Frobenius
functional and a graded orthogonal frame.
The even Frobenius functionals form an open subset of
$(A_+)^*$, of dimension $m$. Over this functional space,
the graded orthogonal frames form a torsor under
$\mathrm O(V_+,b|_{V_+})\times\mathrm O(V_-,b|_{V_-})$,
locally trivial in the \'etale topology and of relative
dimension $\binom m2+\binom n2$.
The action is free by the cyclicity argument used above.
Therefore
\[
  \dim\mathcal U_{\mathrm{even}}
  =
  \dim\mathsf G^{\mathrm{even}}_{r,(m,n)}
  +m+\binom m2+\binom n2.
\]
We now apply \cref{thm:universal} to the ambient commuting
variety and subtract the last three terms, obtaining
\eqref{eq:even-gorenstein-bound}.
Here $m+n\geq3$, so the orthogonal Lie algebra is semisimple;
the case $m+n=4$ is covered by the semisimple form of that theorem.
\end{proof}

For clarity, the parameters removed in the proof are as follows:
\[
 \begin{array}{c|c|c}
 \text{locus}
 &\text{dimension of functional space}
 &\text{dimension of frame group}\\ \hline
 \mathcal H_{r,N}&N&\binom N2\\
 \mathsf G^{\mathrm{odd}}_{r,2n}&n&n^2\\
 \mathsf G^{\mathrm{even}}_{r,(m,n)}
       &m&\binom m2+\binom n2
 \end{array}
\]
In each row, the dimension of the functional space equals
the dimension of the allowed cyclic-vector space. These
contributions cancel, leaving the commuting-variety bound
minus the dimension of the frame group.

For $r\ge3$, the component classification of
\cref{thm:hermitian-max-components} improves the odd-socle estimate by
one.
\begin{corollary}\label{cor:odd-socle-strict}
For $r\ge3$ and $n\ge2$,
\[
 \dim\mathsf G^{\mathrm{odd}}_{r,2n}
 \le\frac{rn(n+1)}2-n^2-1.
\]
\end{corollary}

\begin{proof}
In type $\mathrm{CI}(n)$, a tuple in
$(\mathfrak p_{\mathrm{CI}}^+)^r$ kills $V_+$.
A tuple in $(\mathfrak p_{\mathrm{CI}}^-)^r$ maps $V_+$ into
$V_-$ and kills $V_-$.
In either case, the even part of the cyclic span of
$v\in V_+$ is just $\C v$, so $v$ cannot be cyclic if $n\ge2$.
The even cyclic locus therefore avoids both maximal-dimensional
components. By \cref{thm:hermitian-max-components}, every
other component of the commuting variety has dimension at
most $rn(n+1)/2-1$. After adding the $n$ parameters for
$v\in V_+$, the even cyclic locus therefore has dimension
at most $rn(n+1)/2-1+n$.
Subtracting the $n+n^2$ functional and frame parameters gives
the assertion.
\end{proof}

We now test the estimates of \cref{thm:equivariant-gorenstein} on
explicit families. For a punctual quotient $A=R/I$ of length
$N=\dim_\C A$, we write
$e(A)=\dim_\C(\mathfrak m_A/\mathfrak m_A^2)$ for its embedding dimension.
We call $A$ \emph{curvilinear} if $e(A)\le1$, equivalently if
$A\simeq\C[t]/(t^N)$. For $r=2$ the ordinary estimate
\eqref{eq:gorenstein-bound} reads $N-1$ and is attained: curvilinear
quotients are Gorenstein and form a dense open subset of
$\operatorname{Hilb}^N_0(\mathbb A^2)$, which has dimension $N-1$
\cite{Briancon}. For invariant quotients, the embedding dimension is
controlled by the odd part. The residue field $A/\mathfrak m_A$ is even,
so $A_-\subseteq\mathfrak m_A$, and $\mathfrak m_A/\mathfrak m_A^2$ is
spanned by the images of the coordinates, which are odd; hence
$e(A)\le\dim A_-$. The following proposition shows that the equivariant
estimates are attained in some cases and strict in others, even when the
relevant symmetric pair is Hermitian.

\begin{proposition}\label{prop:gorenstein-sharp}
Let $r\ge2$. The following statements hold:
\begin{enumerate}[label=\textup{(\roman*)}]
\item $\mathcal H_{r,2}=\mathsf G^{\mathrm{odd}}_{r,2}\cong\mathbb P^{r-1}$;
hence \eqref{eq:gorenstein-bound} for $N=2$ and
\eqref{eq:odd-gorenstein-bound} for $n=1$ are attained.
\item For $1\le n\le r$,
$\dim\mathsf G^{\mathrm{even}}_{r,(2,n)}=rn-\binom n2-1$; hence
\eqref{eq:even-gorenstein-bound} is attained for $m=2$.
\item $\dim\mathsf G^{\mathrm{even}}_{r,(3,2)}=2r-2$; hence
\eqref{eq:even-gorenstein-bound} is strict for $(m,n)=(3,2)$ and $r\ge3$.
\item $\dim\mathsf G^{\mathrm{odd}}_{r,4}=2r-2$, so
\cref{cor:odd-socle-strict} is sharp at $r=3$.
\item $\mathsf G^{\mathrm{even}}_{r,(m,2)}=\varnothing$ for $m\ge4$.
\end{enumerate}
\end{proposition}

\begin{proof}
(i) A punctual quotient $A$ of length two has $\dim\mathfrak m_A=1$, so
$\mathfrak m_A^2=0$ by Nakayama's lemma. Hence $\mathfrak m^2\subseteq I$,
and $I$ is determined by the one-dimensional quotient
$\mathfrak m/\mathfrak m^2\to\mathfrak m_A$; these quotients form
$\mathbb P^{r-1}$. Such an ideal is homogeneous, hence invariant, and $A$
is Gorenstein with socle $\mathfrak m_A$, which is odd. Thus
$\mathcal H_{r,2}=\mathsf G^{\mathrm{odd}}_{r,2}\cong\mathbb P^{r-1}$,
and both bounds in question equal $r-1$.

(ii) Let $x_i$ act on $\C[y_1,\ldots,y_r]$ as $\partial/\partial y_i$.
A quadratic form $Q$ of rank $n$ has apolar ideal
$\operatorname{Ann}(Q)
=\{f\in R:f(\partial_{y_1},\ldots,\partial_{y_r})Q=0\}$.
It is homogeneous, hence invariant, and $A=R/\operatorname{Ann}(Q)$ is
Gorenstein with Hilbert function $(1,n,1)$, that is,
$\dim\mathfrak m_A^i/\mathfrak m_A^{i+1}=1,n,1$ for $i=0,1,2$
\cite{IarrobinoKanev}. Its character multiplicities are $(2,n)$, and its
socle $\mathfrak m_A^2$ is even. The line $\C Q$ is recovered from the
degree-two part of $\operatorname{Ann}(Q)$, so distinct projective forms
give distinct ideals. Since the differentiation maps have constant rank
on the rank-$n$ locus, the apolar quotients form an algebraic family of
constant length $n+2$. A rank-$n$ form is a nondegenerate form in the
variables of a unique $n$-dimensional subspace of
$\langle y_1,\ldots,y_r\rangle$, so the projective rank-$n$ forms have
dimension $n(r-n)+\binom{n+1}2-1=rn-\binom n2-1$. This is the value of
\eqref{eq:even-gorenstein-bound} for $m=2$.

(iii) Let $A$ be an invariant Gorenstein quotient with even socle and
character multiplicities $(3,2)$; then $e(A)\le\dim A_-=2$. Suppose
$e(A)=2$. Then $A_-\to\mathfrak m_A/\mathfrak m_A^2$ is an isomorphism,
so $A_-\cap\mathfrak m_A^2=0$. Every monomial of degree three in the
$x_i$ is odd and lies in $\mathfrak m_A^2$, hence vanishes. Thus
$\mathfrak m_A^3=0$ and $\mathfrak m_A^2\subseteq\operatorname{soc}(A)$,
so $\dim A\le1+2+1=4$, a contradiction. Hence $e(A)=1$, and an odd lift
$t$ of a generator of $\mathfrak m_A/\mathfrak m_A^2$ gives
$A\simeq\C[t]/(t^5)$ with $t$ odd. Conversely, this algebra has
character multiplicities $(3,2)$ and even socle $\C t^4$.

Equivariant surjections $R\to\C[t]/(t^5)$ are given by
$x_i\mapsto\alpha_it+\beta_it^3$ with $(\alpha_1,\ldots,\alpha_r)\ne0$.
They form an open subset of $\mathbb A^{2r}$, which carries a flat family
of quotients and hence a morphism to the locus, surjective by the
preceding paragraph. Two surjections have the same kernel exactly when
they differ by a parity-preserving automorphism $t\mapsto ct+dt^3$ with
$c\ne0$, and this two-dimensional group acts freely because the maps are
surjective. The fibre-dimension theorem gives
$\dim\mathsf G^{\mathrm{even}}_{r,(3,2)}=2r-2$, which is less than the
value $3r-4$ of \eqref{eq:even-gorenstein-bound} when $r\ge3$.

(iv) Let $A$ be an invariant Gorenstein quotient of length four with odd
socle, so with character multiplicities $(2,2)$. If $e(A)=2$, then $A$
has Hilbert function $(1,2,1)$, and
$\operatorname{soc}(A)=\mathfrak m_A^2$ is spanned by the images of the
even monomials $x_ix_j$, a contradiction. Hence $A\simeq\C[t]/(t^4)$
with $t$ odd, and the parametrisation of (iii), taken modulo $t^4$,
gives $\dim\mathsf G^{\mathrm{odd}}_{r,4}=2r-2$. For $r=3$ this equals
the bound $4$ of \cref{cor:odd-socle-strict}, whereas
\eqref{eq:odd-gorenstein-bound} gives $5$.

(v) Let $A$ be an invariant Gorenstein quotient of length $N$ with
$\dim A_-=2$. If $e(A)=2$, the argument of (iii) gives $N\le4$. If
$e(A)\le1$, then $A\simeq\C[t]/(t^N)$ with $t$ odd, whose odd part has
dimension $\lfloor N/2\rfloor$; hence $N\le5$. In either case
$m=N-2\le3$.
\end{proof}

\begin{remark}\label{rem:scheme-theoretic-applications}
The cyclic locus can avoid all maximal-dimensional matrix
components, as in \cref{cor:odd-socle-strict}.
Consequently, the Gorenstein bounds need not be equalities.
These arguments concern dimensions and components of reductions;
they do not establish reducedness or local smoothness of the
original moduli schemes.
\end{remark}

\begin{remark}\label{rem:applications-fields}
The results of this section hold over every algebraically closed field
$\kk$ of characteristic zero. The argument of \cref{prop:charzero}
applies, once one notes that the additional data used here are also
defined over $\overline{\mathbb Q}$: the punctual moduli spaces, their
cyclic loci, and the frame bundles.
\end{remark}

\end{document}